\documentclass[12pt]{amsart}

\usepackage{amssymb,amsthm,amsmath}

\usepackage{geometry}
\newtheorem{theorem}{Theorem}[section]
\newtheorem{lemma}[theorem]{Lemma}
\theoremstyle{definition}
\newtheorem{definition}[theorem]{Definition}
\newtheorem{example}[theorem]{Example}
\newtheorem{corollary}[theorem]{Corollary}
\newtheorem{proposition}[theorem]{Proposition}
\theoremstyle{remark}
\newtheorem{remark}[theorem]{Remark}
\numberwithin{equation}{section}

\begin{document}

\title[Global hypoellipticity for pseudo-differential operators]{Global hypoellipticity for a class of \\ pseudo-differential operators on the torus}

\author[F.  de \'{A}vila ]{Fernando de \'{A}vila Silva}
\address{Departamento de Matem\'{a}tica \\ Universidade Federal do Paran\'{a} \\ Caixa Postal 19081 \\ 81531-990, Curitiba, PR, Brazil}
\email{fernando.avila@ufpr.br}

\author[R. Gonzalez]{Rafael Borro Gonzalez}
\address{Departamento de Matem\'{a}tica \\ Universidade Federal do Paran\'{a} \\ Caixa Postal 19081 \\ 81531-990, Curitiba, PR, Brazil}
\email{faelborro@gmail.com}
\thanks{During this work the second author was supported by PNPD/CAPES - Brazil in the Graduate Program in Mathematics, PPGM-UFPR}

\author[A. Kirilov]{\\ Alexandre Kirilov}
\address{Departamento de Matem\'{a}tica \\ Universidade Federal do Paran\'{a} \\ Caixa Postal 19081 \\ 81531-990, Curitiba, PR, Brazil}
\email{akirilov@ufpr.br}

\author[C. de Medeira]{Cleber de Medeira}
\address{Departamento de Matem\'{a}tica \\ Universidade Federal do Paran\'{a} \\ Caixa Postal 19081 \\ 81531-990, Curitiba, PR, Brazil}
\email{clebermedeira@ufpr.br}

\subjclass[2010]{Primary 35B10, 35H10, 35S05}

\date{\today}

\dedicatory{In memory of Todor V. Gramchev}

\begin{abstract}
We show that an obstruction of number-theoretical nature appears as a necessary condition for the global hypoellipticity of the pseudo-differential operator $L=D_t+(a+ib)(t)P(D_x)$ on $\mathbb{T}^1_t\times\mathbb{T}_x^{N}$. This condition is also sufficient when the symbol $p(\xi)$ of $P(D_x)$ has at most logarithmic growth. If $p(\xi)$ has super-logarithmic growth, we show that the global hypoellipticity of $L$ depends on the change of sign of certain interactions of the coefficients with the symbol $p(\xi).$ Moreover, the interplay between the order of vanishing of coefficients with the order of growth of $p(\xi)$ plays a crucial role in the global hypoellipticity of $L$. We also describe completely the global hypoellipticity of $L$ in the case where $P(D_x)$ is homogeneous. Additionally, we explore the influence of irrational approximations of a real number in the global hypoellipticity.
\end{abstract}

\maketitle

\begin{center}
 \begin{minipage}{1.0\textwidth}
   \tableofcontents
 \end{minipage}
\end{center}

\section{Introduction}

We investigate the global hypoellipticity of pseudo-differential operators of the form
\begin{equation}\label{MO}
L=D_t+(a+ib)(t)P(D_x),\quad (t,x)\in\mathbb{T}^1\times\mathbb{T}^N,
\end{equation}
where $a(t)$ and $b(t)$ are real smooth functions on $\mathbb{T}^1,$ and $P(D_x)$ is a pseudo-differential operator of order $m \in \mathbb{R}$ defined on $\mathbb{T}^N\simeq\mathbb{R}^N/(2\pi\mathbb{Z}^N)$. The operator $P(D_x)$ is given by
\begin{equation}\label{pdo}
P(D_x) \cdot u = \sum_{\xi \in \mathbb{Z}^N}{e^{i x \cdot \xi} p(\xi) \widehat{u}(\xi)},
\end{equation}
where $p=p(\xi) \in {S}^m(\mathbb{Z}^N)$ is the toroidal symbol of $P(D_x)$ and
\begin{equation*}
\widehat{u}(\xi) = \dfrac{1}{(2\pi)^N} \int_{\mathbb{T}^N}{e^{- i x \cdot \xi} u(x) dx}, \ \xi \in \mathbb{Z}^N,
\end{equation*}
are the Fourier coefficients of $u$.

The operator  $L$ is said to be \textit{globally hypoelliptic} on $\mathbb{T}^1\times\mathbb{T}^N $ if the conditions $u\in\mathcal{D}'(\mathbb{T}^1\times\mathbb{T}^N)$ and $Lu\in C^{\infty}(\mathbb{T}^1\times\mathbb{T}^N)$ imply that $u\in C^{\infty}(\mathbb{T}^1\times\mathbb{T}^N).$

Even in the case of  vector fields, the investigation of global hypoellipticity on the torus is a challenging problem that still have open questions. Perhaps the question without an answer that is most famous and seemingly far from  a solution is the Greenfield and Wallach conjecture. It states that: if a smoothly closed manifold $M$ admits a globally hypoelliptic vector field $X$, then $M$ is diffeomorphic to a torus and $X$ is smooth conjugated to a Diophantine vector field (see \cite{GW3}).

This conjecture has a geometric version stated in terms of cohomology-free dynamical systems, known as Katok conjecture, and was proved only in some few cases and in dimensions 2 and 3. For more details we refer the works of G. Forni \cite{Fo08}, J. Hounie \cite{HOU82}, and A. Kocsard \cite{Koc09}.

With respect to the differential case of the operator we are interested, with $P(D_x)=D_x$ and $N=1,$ J. Hounie has proved in Theorem 2.2 of \cite{HOU79} that $L=D_t+(a+ib)(t)D_x$ is globally hypoelliptic on $\mathbb{T}^2$ if and only if $b(t)$ does not change sign and either $b_0 \neq 0$ or $a_0$ is an irrational non-Liouville number, where
\begin{equation*}
a_0\doteq(2\pi)^{-1}\int_{0}^{2\pi}a(t)dt \ \textrm { and } \ b_0\doteq(2\pi)^{-1}\int_{0}^{2\pi}b(t)dt.
\end{equation*}

We recall that S. Greenfield and N. Wallach have proved in \cite{GW1} that the above conditions on $a_0$ and $b_0$ means that the constant coefficient operator $D_t+(a_0+ib_0)D_x$ is globally hypoelliptic. Therefore, the global hypoellipticity of $D_t+(a_0+ib_0)D_x$ is a necessary condition for the global hypoellipticity of the operator with variable coefficients $D_t+(a+ib)(t)D_x$.

We prove that this necessity remains valid for any pseudo-differential operator $P(D_x)$ defined on the $N$-dimensional torus, that is, if the operator $L$ defined in \eqref{MO} is globally hypoelliptic then the constant coefficient operator
\begin{equation}\label{MOCC}
L_0=D_t+(a_0+ib_0)P(D_x),\quad (t,x)\in\mathbb{T}^1\times\mathbb{T}^N,
\end{equation}
is also globally hypoelliptic (see Theorem \ref{ncm2}).

We also show that the global hypoellipticity of $L_0$ and the control of the sign of the imaginary part of the functions
\[
  t\in\mathbb{T}^1\mapsto \mathcal{M}(t,\xi)\doteq(a+ib)(t)p(\xi), \ \xi \in \mathbb{Z}^N,
\]
for sufficiently large $|\xi|$, are sufficient conditions to the global hypoellipticity of $L$ (see Theorem \ref{gt1}).

Although the global hypoellipticity of $L_0$ cannot be removed in the study of the global hypoellipticity of $L$, the converse of Theorem \ref{gt1} in general does not hold; unlike the differential case $P(D_x)=D_x$. In Sections \ref{sectmlg} and \ref{sectfr} we exhibit examples of globally hypoelliptic operators in which the imaginary part of the functions $t\in\mathbb{T}^1 \mapsto \mathcal{M} (t,\xi)$ changes sign for infinitely many indexes $\xi\in\mathbb{Z}^N$ (see Examples \ref{exampcc}, \ref{exampcc2}, \ref{exampsign1}, \ref{exampabncs}, and the first example in Subsection \ref{exar2}).

We point out that our results are not a consequence of the Hounie's abstract results in \cite{HOU79}, even when our operator fits in the conditions assumed in that work. Depending on $P(D_x)$, the scales of Sobolev spaces used by Hounie are different from the usual Sobolev spaces, which implies in a different notion of global hypoellipticity. We refer the reader to \cite{AGK19}, Section 3.3, for more details.

In Section \ref{sectcc} we study operators with constant coefficients giving a special attention to the case when $P(D_x)$ is a  homogeneous operator of rational degree $m$ on $\mathbb{T}^1$, see Theorem \ref{gt0}. In this case, our main contribution is to shed light on the connections between hypoellipticity and certain approximations of real numbers, which are not considered in \cite{GW1}. Indeed, in our approach the global hypoellipticity depends on the following approximations
\begin{align*}
\left | \frac{\tau}{|\xi|^{m}} + p\left (\pm 1\right) \right|, \ (\tau,\xi) \in \mathbb{Z}\times\mathbb{Z}_{\ast}
\end{align*}
where the numbers $\tau/{|\xi|^{m}}$ can be irrational, depending on $m.$

As a consequence of Theorem \ref{gt0}, if $m=\ell/q$ is irreducible, with $\ell,q\in\mathbb{N}$, then the operator $D_t+\alpha(D_x^2)^{m/2}$ is globally hypoelliptic if and only if $\alpha^q$ is an irrational non-Liouville number. Notice that the global hypoellipticity of this operator does not depend on $\ell.$ For example,  $D_t+\sqrt{2}(D_x^2)^{\ell/2q}$ is globally hypoelliptic if and only if $q$ is odd.

Regarding the case of variable coefficients, one of the contributions of this work is to show that the global hypoellipticity of the operator $L$ defined in \eqref{MO} is related to the growth of the real and imaginary parts of the symbol $p(\xi)$ when $|\xi| \rightarrow \infty$.

In Section \ref{sectmlg}, we give a complete characterization for the global hypoellipticity of $L$ when either $\alpha(\xi)$ or $\beta(\xi)$ has at most logarithmic growth, where
\[
  p(\xi) = \alpha(\xi)+i\beta(\xi), \ \xi \in \mathbb{Z}^N.
\]

When both $\alpha(\xi)$ and $\beta(\xi)$ have at most logarithmic growth we show that the change of sign of the functions
\begin{equation*}
t \in \mathbb{T}^1 \mapsto \Im\mathcal{M}(t,\xi) = a(t)\beta(\xi) + b(t)\alpha(\xi), \ \xi \in \mathbb{Z}^N,
\end{equation*}
does not play any role in the global hypoellipticity of $L$. More precisely, we prove that $L$, defined in \eqref{MO}, is globally hypoelliptic if and only if $L_0$, defined in \eqref{MOCC} is globally hypoelliptic. This equivalence comes from the reduction to normal form, that is a technique well explored in the works \cite{AGK19,CC00,CG04,DGY97,DGY02,Petr11}.

If $\beta(\xi)$ has at most logarithmic growth, but $\alpha(\xi)$ has super-logarithmic growth, then we prove that $L$ is globally hypoelliptic if and only if $L_0$ is globally hypoelliptic  and $b(t)$ does not change sign. This result remains valid if we exchange $\beta(\xi)$ by $\alpha(\xi)$ and $b(t)$ by $a(t)$, see  Subsection \ref{ssectsuplg}.

When $\alpha(\xi)$ and $\beta(\xi)$  have super-logarithmic growth, the interactions between the functions $a(t)\beta(\xi)$ and $b(t)\alpha(\xi)$ play a larger role. In this case, the operator $L$ may be non-globally hypoelliptic even if $L_0$ is globally hypoelliptic and both $a(t)$ and $b(t)$ do not change sign (see Examples \ref{exampabncs}, \ref{exampnew1} and the second example in Subsection \ref{exar2}).

On the other hand, $L$ may be globally hypoelliptic even if both $a(t)$ and $b(t)$ changes sign provided that $\alpha(\xi)$ and $\beta(\xi)$ go to infinity with the same order of growth (see Example \ref{exampnew2}).

In the case where both the parts $\alpha(\xi)$ and $\beta(\xi)$ have super-logarithmic growth and $\alpha(\xi)/\beta(\xi) \rightarrow K,$ as $|\xi|\rightarrow\infty,$ we show that $L$ is not globally hypoelliptic if the function \[t\in\mathbb{T}^1\rightarrow a(t)+b(t)K\] changes sign (Corollary \ref{cornew}). In particular, if $p(\xi)$ has super-logarithmic growth with $\alpha(\xi)=o(\beta(\xi)),$ then $L$ is not globally hypoelliptic if $a(t)$ changes sign. Analogously, $L$ is not globally hypoelliptic when $\beta(\xi)=o(\alpha(\xi))$ and $b(t)$ changes sign (see Corollary \ref{cornew2}).

Another contribution we give is to present (in Subsection \ref{exar2}) a relation between the global hypoellipticity of the operator $L$ and the order of vanishing of the coefficients $a(t)$ and $b(t)$. We emphasize that this phenomenon is more common in the study of the global solvability of vector fields on the torus (see \cite{BCP,BD,BP,G}).

In Section \ref{sectsc}, we describe completely the global hypoellipticity of $L$ in the case where $P(D_x)$ is  homogeneous (see Theorem \ref{corolph} and Corollary \ref{inttheor}). For these operators the converse of Theorem \ref{gt1} holds. Moreover, we analyze the case of sums of  homogeneous operators extending Theorem 1.3 of \cite{BDG}, see Corollary \ref{lastcor}.

For more results on the problem of global hypoellipticity and global solvability of equations and systems of equations on the torus we refer the reader to the works  \cite{BERG99,BCM,BK07,BKNZ12,BKNZ15,BdMZ12,GPY92,GPY93,HP00} and the references therein.

\section{The constant coefficient operators}\label{sectcc}

By following the approach introduced by Greenfield and Wallach in \cite{GW1}, we may cha\-racterize the global hypoellipticity of the operator
\begin{equation}\label{LCC}
L = D_t + P(D_x), \ (t, x) \in \mathbb{T}^1 \times \mathbb{T}^{N},
\end{equation}
by means of a control in its symbol
$$L(\tau, \xi) = \tau + p(\xi), \ (\tau,\xi)\in \mathbb{Z}\times \mathbb{Z}^N.$$

\begin{theorem}\label{t-2.1}
The operator $L$ in \eqref{LCC}  is globally hypoelliptic if and only if there exist positive constants $C,$ $M$ and $R$ such that
\begin{equation*}
\quad|\tau + p(\xi)|\geqslant  \frac{C}{(|\tau|+|\xi|)^M}, \quad \textrm{for all} \quad |\tau|+ |\xi| \geqslant  R.
\end{equation*}
\end{theorem}

The proof of this result follows the same ideas of the differential case made in \cite{GW1}.

Note that, if the imaginary part of $p(\xi)$ does not approach to zero rapidly, then the estimate in Theorem \ref{t-2.1} is verified. More precisely, if there exists $M\geq0$ such that \[\label{eqq0}\displaystyle\liminf_{|\xi|\rightarrow\infty}|\xi|^M|\Im p(\xi)|>0,\] then the operator $L=D_t+P(D_x)$ is globally hypoelliptic.

This type of condition appears in Theorem 5.3 of \cite{DGY02}, where the authors studied the relation between global hypoellipticity and simultaneous inhomogeneous Siegel conditions.

On the other hand, without this control in the imaginary part, for instance when $\Im p(\xi)\equiv 0$, Diophantine phenomena appear. When the symbol is homogeneous of rational positive degree, we present a new relation between global hypoellipticity and Liouville numbers in  Theorem \ref{gt0}.

We observe that each toroidal symbol $p$ in the class  $S^{m}(\mathbb{Z}^N)$ can be extended to an Euclidean symbol $\widetilde{p} \in S^{m}(\mathbb{R}^N)$ such that $p = \widetilde{p}\, |_{\mathbb{Z}^N}$ (see Theorem  4.5.3 of  \cite{RT3}).

\begin{definition}\label{defphs}
We say that a toroidal symbol $p(\xi)$ is  homogeneous of degree $m$ if it has an Euclidean extension $\widetilde{p}(\xi)$ such that
\begin{equation*}
p(\xi)=|\xi|^m \widetilde{p}(\xi/|\xi|), \ \xi \in \mathbb{Z}^N_*.
\end{equation*}
\end{definition}

In order to not overload our notation, we will use the notation $p(\xi/|\xi|)$ in place of $\widetilde{p}(\xi/|\xi|)$.

When the symbol $p(\xi)$ is  homogeneous of degree $m$, it follows from Theorem \ref{t-2.1} that the operator  $L$ given by \eqref{LCC} is globally hypoelliptic if and only if there exist positive constants $C,$ $M,$ and $R,$ such that
\begin{align}\label{L-hom-1}
\left | \frac{\tau}{|\xi|^{m}} + p\left (\frac{\xi}{|\xi|}\right) \right|\geqslant  \frac{C}{(|\tau|+|\xi|)^M},
\end{align}
for all $(\tau,\xi)\in\mathbb{Z}\times\mathbb{Z}_{\ast}^N$ which satisfy $|\tau| + |\xi| \geqslant  R.$

\subsection{Global hypoellipticity and Liouville numbers}

Let $P(D_x)$  be an operator on $\mathbb{T}^1$ with symbol $p(\xi)$  homogeneous of degree $m$. In this case
$$p(\xi)=|\xi|^mp(\pm 1), \ \mbox{ for all } \xi\in\mathbb{Z}_*.$$

Thus, when $m<0$, by using condition \eqref{L-hom-1} we see that the operator $L$ is globally hypoelliptic if and only if $p(\pm 1)\neq 0$. Similarly, when $m=0$, it follows that  $L$ is globally hypoelliptic if and only if $p(\pm 1)\notin  \mathbb{Z}$.

The case in which $m$ is a positive rational number and $\Im p(\pm1)=0,$  is more interesting. We now move to describe it.

By using the notations
\[
  p(1) = \alpha + i \beta \quad \textrm{and} \quad p(-1) = \widetilde{\alpha} + i \widetilde{\beta},
\]
we have
\begin{equation}\label{eqq}
\left | \frac{\tau}{|\xi|^{m}} + p\left (\frac{\xi}{|\xi|}\right) \right| =
\left \{
\begin{array}{l}
\left| \dfrac{\tau}{\xi^{m}} + (\alpha + i \beta) \right|,  \ \textrm{ if } \xi>0, \vspace{2mm} \\
\left| \dfrac{\tau}{|\xi|^{m}} + (\widetilde{\alpha}  + i \widetilde{\beta}) \right|, \ \textrm{ if } \ \xi<0.
\end{array}
\right.
\end{equation}

In this case, when $\beta = 0$ (respectively $\widetilde{\beta}=0$) we must control the approximations of the real number $\alpha$  (respectively $\widetilde{\alpha}$) by numbers of the type ${\tau}/{|\xi|^{m}},$ for all $(\tau, \xi) \in \mathbb{Z} \times \mathbb{Z}_{\ast}.$

\begin{definition}
An irrational number $\lambda$ is said to be a Liouville number if there exists a sequence $(j_n,k_n)\in\mathbb{Z}\times\mathbb{N},$ such that $k_n \rightarrow \infty$ and
\begin{equation*}
\left|\lambda-\dfrac{j_n}{k_n}\right|<(k_n)^{-n}, \ n\in\mathbb{N}.
\end{equation*}
\end{definition}

Under the previous notation we have the following result:

\begin{theorem}\label{gt0} If $p=p(\xi)$ is a  homogeneous symbol of degree $m=\ell/q$ with  $\ell, q\in\mathbb{N},$ and $\gcd(\ell,q)=1$,  then the operator
   $$L=D_t+P(D_x),\ (t,x)\in\mathbb{T}^2,$$
is globally hypoelliptic if and only if $\alpha^q$ is an irrational non-Liouville number whenever $\beta=0,$ and $\widetilde{\alpha}^q$ is an irrational non-Liouville number whenever $\widetilde{\beta}=0$.
\end{theorem}

\begin{proof}
 Inequality \eqref{L-hom-1} is easily verified when $\beta \cdot \widetilde{\beta} \neq 0$, consequently $L$ is globally hypoelliptic in this case. Therefore, in order to prove Theorem \ref{gt0} it is enough to consider either $\beta=0$ or $\widetilde\beta=0$.

We start by considering the case $\beta=0 $ and $\widetilde\beta\neq0$. In this situation, it follows from \eqref{L-hom-1} and \eqref{eqq} that $L$ is globally hypoelliptic if and only if  there exist positive constants $C,$ $M$ and $R$ such that
\begin{equation}\label{eqqq1}
\left|\dfrac{\tau}{\xi^m}+\alpha\right|\geqslant C(|\tau|+\xi)^{-M},
\end{equation}
for all $(\tau,\xi)\in\mathbb{Z}\times\mathbb{N}$ such that $|\tau|+\xi\geqslant R.$

 Since $\beta=0$, if $\alpha=0$ then $p(\xi)=0$ for all $\xi>0$ and, therefore, $L$ is not globally hypoelliptic. From now on, without loss of generality, we assume that $\alpha>0.$

When $\alpha^{q}$ is a rational number, we prove that $L$ is not globally hypoelliptic by exhibiting infinitely many $(\tau,\xi)\in\mathbb{Z}\times\mathbb{N}$ such that $\left | {\tau}/{\xi^{m}} -\alpha \right|=0.$

We then write $\alpha^q=\widetilde{p}/\widetilde{q}$, with $\widetilde{p},\widetilde{q}\in\mathbb{N}$.  By prime factorization we have
\begin{equation*}
\widetilde{q}=q_1^{\gamma_1}\cdots q_r^{\gamma_r} \  \textrm{ and } \    \widetilde{p}=p_1^{\sigma_1}\cdots p_s^{\sigma_s}.
\end{equation*}

Since $\gcd(\ell,q)=1,$ there exists $(x_i,y_i)\in\mathbb{N}^2$ and $(v_j,w_j)\in\mathbb{N}^2$ such that
\begin{equation*}
\ell x_i-qy_i=\gamma_i, \ i=1,\ldots,r \ \textrm{ and } \ qv_j-\ell w_j=\sigma_j, \ j=1,\ldots,s.
\end{equation*}

Define
\begin{equation*}
\tau_n=n^\ell q_1^{y_1}\cdots q_r^{y_r}p_1^{v_1}\cdots p_s^{v_s} \  \textrm{ and } \    \xi_n=n^q q_1^{x_1}\cdots q_r^{x_r}p_1^{w_1}\cdots p_s^{w_s}.
\end{equation*}

It follows that $\widetilde{q}\tau_n^q=\widetilde{p}\xi_ n^\ell,$ for all $n\in\mathbb{N};$ hence
\begin{equation*}
\dfrac{\tau_n}{(\xi_n)^{m}}=\dfrac{\tau_n}{(\xi_n)^{\ell/q}}=\left(\, \frac{\,\widetilde{p}\,}{\,\widetilde{q}\,} \,\right)^{1/q}=\alpha, \  \textrm{for all} \ n\in\mathbb{N},
\end{equation*}
and then $L$ is not globally hypoelliptic.

From now on, assume that $\alpha^q$ is an irrational number.

If $L$ is not globally hypoelliptic,
it follows from \eqref{eqqq1} that there exists a sequence $(\tau_n,\xi_n)\in\mathbb{Z}\times\mathbb{N}$ such that
\[
\left|\dfrac{\tau_n}{\xi_n^{\ell/q}}-\alpha\right|<(|\tau_n|+\xi_n)^{-n},\quad |\tau_n|+\xi_n\geqslant n.
\]

By taking  $j_n=-\tau_n^q$ and $k_n=\xi_n^\ell$ we obtain \[\left|\dfrac{j_n}{k_n}+\alpha^{q}\right|=\left|-\left(\dfrac{\tau_n}{\xi_n^{\ell/q}}\right)^q+\alpha^{q}\right|=
\left|\dfrac{\tau_n}{\xi_n^{\ell/q}}-\alpha\right|\cdot\left|\displaystyle\sum_{j=1}^{q}\left(\dfrac{\tau_n}{\xi_n^{\ell/q}}\right)^{q-j}\alpha^{j-1}\right|.\]

Since $\displaystyle\sum_{j=1}^{q}({\tau_n}/{\xi_n^{\ell/q}})^{q-j}\alpha^{j-1}$ goes to $ q\alpha^{q-1},$ as $n$ goes to infinity, it follows that
\[\left|\dfrac{j_n}{k_n}+\alpha^q\right|\leqslant C\left|\dfrac{\tau_n}{\xi_n^{\ell/q}}-\alpha\right|<
C(|\tau_n|+\xi_n)^{-n}\leqslant  Ck_n^{-n/\ell},\quad \textrm{for all} \ n,\]
where the constant $C>0$ does not depend on $j_n$ and $k_n$.

The estimate above implies that $\alpha^q$ is a Liouville number.

\smallskip
Assuming that $\alpha^{q}$ is a Liouville number, let us show that $L$ is not globally hypoelliptic.
Indeed, if $\alpha^{q}$ is a Liouville number, then there is a sequence $(j_n,k_n)\in\mathbb{N}^2,$ $j_n+k_n\geqslant n$, such that
\begin{equation*}
|j_n-\alpha^q k_n|<(j_n+k_n)^{-n}.
\end{equation*}
By multiplying this inequality by
\begin{equation}\label{divinmult}
j_n^{(q-1)\ell+(\ell-1)\widetilde{p}\ell}\,k_n^{\widetilde{q}q},
\end{equation}
where $\widetilde{p}$ and $\widetilde{q}$ are positive integers such that $\widetilde{p}\ell-\widetilde{q}q=1$, we obtain
\[|(j_n^{\ell+\widetilde{q}(\ell-1)}k_n^{\widetilde{q}})^{q}-\alpha^q(k_n^{\widetilde{p}}j_n^{q-1+(\ell-1)\widetilde{p}})^\ell|<j_n^{(q-1)\ell+(\ell-1)\widetilde{p}\ell}k_n^{\widetilde{q}q}(j_n+k_n)^{-n}.\]

Suppose, by contradiction, that $L$ is globally hypoelliptic. By \eqref{eqqq1}, there exist positive constants $C$, $M$ and $R$ such that
$$|\tau-\alpha\xi^{\ell/q}|\geqslant C(|\tau|+\xi)^{-M},$$
for all $(\tau,\xi)\in\mathbb{N}\times \mathbb{N}$ such that $\tau+\xi>R.$

Since
\begin{eqnarray*}
|\tau^q-\alpha^q\xi^\ell|&=&|\tau-\alpha\xi^{\ell/q}|\cdot\left|\sum_{\kappa=1}^{q}\tau^{q-\kappa}(\alpha\xi^{\ell/q})^{\kappa-1}\right| \\[2mm]
   &\geqslant& C (\tau+\xi)^{-M} \left|\sum_{\kappa=1}^{q}\tau^{q-\kappa}(\alpha\xi^{\ell/q})^{\kappa-1}\right|,
\end{eqnarray*}
and
\begin{equation*}
\left|\sum_{\kappa=1}^{q}\left(j_n^{\ell+\widetilde{q}(\ell-1)}k_n^{\widetilde{q}}\right)^{q-\kappa}\alpha^{\kappa-1}\left(k_n^{\widetilde{p}}j_n^{q-1+(\ell-1)\widetilde{p}}\right)^{\frac{\ell(\kappa-1)}{q}}\right|\geqslant \widetilde{C},\quad (j_n\geqslant 1,\,\, k_n\geqslant 1),
\end{equation*}
for some $\widetilde{C}>0$, it follows that
\begin{equation*}
 C\widetilde{C}\leqslant \left (j_n^{\ell+\widetilde{q}(\ell-1)}k_n^{\widetilde{q}}+k_n^{\widetilde{p}}j_n^{q-1+(\ell-1)\widetilde{p}}\right)^{M}j_n^{(q-1)\ell+(\ell-1)\widetilde{p}\ell}k_n^{\widetilde{q}q}(j_n+k_n)^{-n},
\end{equation*}
 for all $n\in\mathbb{N}.$

Now, by taking
\begin{equation*}
K=\max\{\ell+\widetilde{q}(\ell-1), \widetilde{q} ,\widetilde{p}, q-1+(\ell-1)\widetilde{p},  (q-1)\ell+(\ell-1)\widetilde{p}\ell, \widetilde{q}q \}
\end{equation*}
we obtain
\begin{eqnarray*}
0<C\widetilde{C}&\leqslant & (j_n^{\ell+\widetilde{q}(\ell-1)}k_n^{\widetilde{q}}+k_n^{\widetilde{p}}j_n^{q-1+(\ell-1)\widetilde{p}})^{M} j_n^{(q-1)\ell+ (\ell-1)\widetilde{p}\ell} k_n^{\widetilde{q}q}(j_n+k_n)^{-n}\\[4mm]
&\leqslant&(j_n^{K}k_n^{K}+k_n^Kj_n^{K})^{M}j_n^{K}k_n^{K}(j_n+k_n)^{-n}\\[4mm]
&=&2^M(j_n+k_n)^{-n+2K(M+1)},
\end{eqnarray*}
for all $n\in\mathbb{N},$ which  is a contradiction, since the right-hand side goes to zero as $n$ goes to infinity.

Finally, in the case in which $\beta\neq0$ and $\widetilde{\beta}=0,$ a slight modification in the previous arguments give us that $L$ is globally hypoelliptic if and only if $\widetilde{\alpha}^q$ is an irrational non-Liouville number. Analogously, if $\beta=0$ and $\widetilde{\beta}=0,$ then $L$ is globally hypoelliptic if and only if both $\alpha^q$ and $\widetilde{\alpha}^q$ are irrational non-Liouville numbers.

\end{proof}

As consequence of Theorem \ref{gt0} we obtain the following examples.

\begin{example}
Let $m=\ell/q$ be a positive rational number with $gcd(\ell,q)=1$, then  $L=D_t+\alpha (D_x^2)^{m/2}$ is globally hypoelliptic if and only if
$\alpha^q$ is an irrational non Liouville number. In particular, for the non-Liouville number $\alpha=\sqrt{2}$ the operator $L=D_t+\alpha(D_x^2)^{1/2}$ is globally hypoelliptic while   $L=D_t+\alpha(D_x^2)^{1/4}$ is not.
\end{example}

\begin{example}
Let  $\lambda=\sum_{n=1}^{\infty}10^{-n!}$ be the Liouville constant. For each integer $q\geqslant 2$ we have that ${\lambda}^q \, 3/2$ is a Liouville number while $\lambda \, \sqrt[q]{3/2}$  is not (see \cite{DJ}). Therefore, by taking  $\alpha=\lambda\, \sqrt[q]{3/2}$  we have that $L=D_t+\alpha (D_x^2)^{1/2q}$ is not globally hypoelliptic for each integer $q\geqslant  2$.
\end{example}

\section{The variable coefficient operators}\label{sectgr}

In this section we study the global hypoellipticity of the operator \eqref{MO}, which we recall
\[L=D_t+(a+ib)(t)P(D_x),\quad (t,x)\in\mathbb{T}^1\times\mathbb{T}^N,\]
where $a(t)$ and $b(t)$ are real valued smooth functions on $\mathbb{T}^1$ and $P(D_x)$ is a pseudo-differential on $\mathbb{T}^N$ with symbol $p=p(\xi),$ $\xi\in\mathbb{Z}^N.$

Without any assumption about the behavior of $p(\xi),$ as $|\xi|\rightarrow\infty,$ we will present a necessary condition and, also, sufficient conditions for the global hypoellipticity of $L.$

First, we show that the global hypoellipticity of $L_0=D_t+(a_0+ib_0)P(D_x),$ where \[a_0=(2 \pi)^{-1} \int_{0}^{2 \pi} a(t) dt\quad\textrm{and}\quad b_0=(2 \pi)^{-1} \int_{0}^{2 \pi} b(t) dt,\] is necessary for the global hypoellipticity of $L$ (Theorem \ref{ncm2}). After this, we will show that this condition is also sufficient provided that the imaginary part of the function
\begin{equation*}
t\in\mathbb{T}^1  \mapsto {\mathcal{M}}(t, \xi) \doteq  (a+ib)(t)p(\xi), \ \xi \in \mathbb{Z}^N,
\end{equation*}
does not change sign, for all $|\xi|$ large enough (Theorem \ref{gt1}).

By using partial Fourier series in the variable $x,$ we can write a distribution $u$ in $\mathcal{D}' (\mathbb{T}^1\times\mathbb{T}^{N})$ as
\begin{equation*}
u = \sum_{\xi \in \mathbb{Z}^N} \widehat{u}(t, \xi) e^{i x \xi},
\end{equation*}
where $\widehat{u}(t,\xi) = (2\pi)^{-N} \langle u(t, \cdot) , e^{-i x \cdot \xi}\rangle.$ Hence, the equation $(i L) u =f$ lead us to consider the differential equations
\begin{equation}\label{diff-eq}
\partial_t \widehat{u}(t, \xi) +    i{\mathcal{M}}(t, \xi) \widehat{u}(t,\xi) = \widehat{f}(t, \xi), \ t \in \mathbb{T}^1, \ \textrm{for all} \ \xi \in \mathbb{Z}^N.
\end{equation}

With the notations \[{\mathcal{M}}_0(\xi) = (2 \pi)^{-1} \int_{0}^{2 \pi} {\mathcal{M}}(t, \xi) dt=(a_0+ib_0)p(\xi)\] and
\begin{equation}\label{ZM-set}
Z_{{\mathcal{M}}} = \{\xi \in \mathbb{Z}^N; \, {\mathcal{M}}_0(\xi) \in\mathbb{Z} \},
\end{equation}
we have:

\begin{lemma}\label{lt1}If $u\in \mathcal{D}'(\mathbb{T}^1\times\mathbb{T}^N)$ and $iLu=f \in C^{\infty}(\mathbb{T}^1\times\mathbb{T}^N),$ then equation \eqref{diff-eq} implies that $\widehat{u}(\cdot,\xi)$ belongs to $C^{\infty}(\mathbb{T}^1),$ for all $\xi\in \mathbb{Z}^N.$ Moreover, for each $\xi\notin Z_{\mathcal{M}},$ equation \eqref{diff-eq} has a unique solution, which can be written in the following two ways:
\begin{align}
\widehat{u}(t, \xi) = \frac{1}{1 - e^{-  2 \pi i{\mathcal{M}_0}(\xi)}} \int_{0}^{2\pi}\exp\left(-i\int_{t-s}^{t}\!\!{\mathcal{M}}(r, \, \xi) \, dr\right) \widehat{f}(t-s, \xi)ds, \label{Solu-1}
\end{align}
or
\begin{align}
\widehat{u}(t, \xi) = \frac{1}{e^{ 2 \pi i{\mathcal{M}_0}(\xi)} - 1} \int_{0}^{2\pi}\exp\left(i\int_{t}^{t+s}\!\!{\mathcal{M}}(r, \, \xi) \, dr\right) \widehat{f}(t+s, \xi)ds. \label{Solu-2}
\end{align}
\end{lemma}

Furthermore, we have the following characterization for the global hypoellipticity of $L_0$.

\begin{proposition}\label{lt2} The following statements are equivalent:
\begin{itemize}
\item[i)] $L_0$ is globally hypoelliptic;
\item[ii)]  There exist positive constants $C,$ $M,$ and $R$ such that
\[|\tau+\mathcal{M}_0(\xi)|\geqslant C(|\tau| + |\xi|)^{-M}, \ \textrm{for all} \ (|\tau| + |\xi|) \geqslant R;\]
\item[iii)] There exist positive constants $\widetilde{C},$ $\widetilde{M},$ and $\widetilde{R}$ such that
\[|1-e^{\pm2\pi i\mathcal{M}_0(\xi)}|\geqslant \widetilde{C}|\xi|^{-\widetilde{M}}, \ \textrm{for all} \  |\xi|\geqslant \widetilde{R}.\]
\end{itemize}
\end{proposition}

The equivalence $i)\Leftrightarrow ii)$ follows from Theorem \ref{t-2.1} and the equivalence $ii)\Leftrightarrow iii)$
is a technical result that is a slight modification of the proof of Lemma 3.1 of \cite{BDG}.

\subsection{A necessary condition}

Our first result in this section is the following:

\begin{proposition}\label{tt1}
If $L$ is globally hypoelliptic, then the set $Z_{\mathcal{M}}$ defined in \eqref{ZM-set} is finite.
\end{proposition}

\begin{proof} If $Z_{\mathcal{M}}$ is infinite, then there exists a sequence $\{\xi_n\}$ such that $|\xi_n|$ is increasing and $\mathcal{M}_0(\xi_n)\in \mathbb{Z}.$
Set
\[c_{n}=\exp\left({-\int_{0}^{t_n}\Im\mathcal{M}(r,\xi_n)dr}\right),\]
where $t_n\in[0,2\pi]$ is such that
\[\int_{0}^{t_n}\Im\mathcal{M}(r,\xi_n)dr=\max_{t\in[0,2\pi]}\int_{0}^{t}\Im\mathcal{M}(r,\xi_n)dr.\]
For each $\xi_n$ the function
\[\widehat{u}(t,\xi_n)=c_n \exp\left({-i\int_{0}^{t}\mathcal{M}(r,\xi_n)dr}\right)\]
is smooth on $\mathbb{T}^1$ and satisfies the equation
\begin{equation*}
\partial_t\widehat{u}(t,\xi_n)+i\mathcal{M}(t,\xi_n)\widehat{u}(t,\xi_n)=0.
\end{equation*}

Moreover, $|\widehat{u}(t,\xi_n)|\leqslant 1,$ for all $t\in[0,2\pi],$ and $|\widehat{u}(t_{n},\xi_n)|=1.$ Hence,
\begin{equation*}
u=\sum_{n=1}^{\infty}\widehat{u}(t,\xi_n) e^{ix\xi_n} \in \mathcal{D}'(\mathbb{T}^1\times\mathbb{T}^N)\setminus C^{\infty}(\mathbb{T}^1\times\mathbb{T}^N),
\end{equation*}
and satisfies $Lu=0.$ Therefore, $L$ is not globally hypoelliptic.
\end{proof}

\begin{remark}\label{remarkcc} In the case of constant coefficients the previous result implies that $L_0$ is not globally hypoelliptic if $Z_\mathcal{M}$ is infinite. Therefore, every time we assume that $L_0$ is globally hypoelliptic it is understood that $Z_\mathcal{M}$ is finite.
\end{remark}

Now we present our first general result on global hypoellipticity.

\begin{theorem}\label{ncm2}If $L$ is globally hypoelliptic, then $L_0$ is globally hypoelliptic.
\end{theorem}

\begin{proof}We assume that $L_0$ is not globally hypoelliptic and prove that $L$ is not globally hypoelliptic.

By Proposition \ref{lt2}, there is a sequence $\{\xi_n\}$ such that $|\xi_n|$ is strictly increasing, $|\xi_n|>n,$ and \begin{equation}\label{est5}
|1-e^{-2\pi i\mathcal{M}_0(\xi_n)}|<|\xi_n|^{-n}, \ \textrm{for all} \ n\in\mathbb{N}.
\end{equation}

By Proposition \ref{tt1}, it is enough to consider the case where $Z_{\mathcal{M}}$ is finite and $\xi_n\not\in Z_{\mathcal{M}},$ for all $n.$

For each $n,$ we may choose $t_n\in[0,2\pi]$ so that $\int_{t_n}^{t}\Im\mathcal{M}(r,\xi_n)dr\leqslant 0,$ for all $t\in[0,2\pi].$

Indeed, for all $t\in[0,2\pi]$ we write
\[\int_{t_n}^{t}\Im\mathcal{M}(r,\xi_n)dr=\int_{0}^{t}\Im\mathcal{M}(r,\xi_n)dr-\int_{0}^{t_n}\Im\mathcal{M}(r,\xi_n)dr,\]
and it is enough to consider $t_n$ satisfying \[\int_{0}^{t_n}\Im\mathcal{M}(r,\xi_n)dr=\max_{t\in[0,2\pi]}\int_{0}^{t}\Im\mathcal{M}(r,\xi_n)dr.\]

By passing to a subsequence, we may assume that there exists $t_{0}\in[0,2\pi]$ such that $t_n\rightarrow t_0,$ as $n\rightarrow\infty.$

Let $I$ be a closed interval in $(0,2\pi)$ such that $t_0\not\in I.$ Consider $\phi$ belonging to $C^{\infty}_{c}(I,\mathbb{R}),$ such that $0\leqslant \phi(t)\leqslant 1$ and
$\int_{0}^{2\pi}\phi(t)dt>0.$

For each $n,$ we define $\widehat{f}(\cdot,\xi_n)$ as being the $2\pi-$periodic extension of
\[(1-e^{-2\pi i\mathcal{M}_0(\xi_n)})\exp\left(-\int_{t_n}^{t}i\mathcal{M}(r,\xi_n)dr\right)\phi(t).\]

Since $p(\xi)$ increases slowly, $\int_{t_n}^{t}\Im\mathcal{M}(r,\xi_n)dr\leqslant 0$ for all $t\in[0,2\pi],$ and since \eqref{est5} holds, it follows that $\widehat{f}(\cdot,\xi_n)$ decays rapidly. Hence, \[f(t,x)=\sum_{n=1}^{\infty}\widehat{f}(t,\xi_n)e^{ix\xi_n} \in C^{\infty}(\mathbb{T}^1\times\mathbb{T}^N).\]

In order to exhibit a distribution $u\in\mathcal{D}'(\mathbb{T}^1\times\mathbb{T}^N)\setminus C^{\infty}(\mathbb{T}^1\times\mathbb{T}^N)$ such that $iLu=f,$ we consider
\[\widehat{u}(t,\xi_n)=\frac{1}{ 1-e^{-2\pi i\mathcal{M}_0(\xi_n)}}\int_{0}^{2\pi}\exp\left(-\int_{t-s}^{t}i\mathcal{M}(r,\xi_n)dr\right)\widehat{f}(t-s,\xi_n)ds.\]

Note that $1-e^{-2\pi i\mathcal{M}_0(\xi_n)}\neq0,$ since $\xi_n\not\in Z_{\mathcal{M}},$ and $\widehat{u}(\cdot,\xi_n)\in C^{\infty}(\mathbb{T}^1)$ (Lemma \ref{lt1}). Moreover, for $t,s\in[0,2\pi]$ such that $t-s \geqslant 0,$ we have

\begin{align*}
&\left|\frac{1}{ 1-e^{-2\pi i\mathcal{M}_0(\xi_n)}}\exp\left(-\int_{t-s}^{t}i\mathcal{M}(r,\xi_n)dr\right)\widehat{f}(t-s,\xi_n)\right| \leqslant \\[2mm]
& \,\, \exp\left(\int_{t-s}^{t}\Im\mathcal{M}(r,\xi_n)dr+\int_{t_n}^{t-s}\Im\mathcal{M}(r,\xi_n)dr\right) =\, \exp\left(\int_{t_n}^{t}\Im\mathcal{M}(r,\xi_n)dr\right),
\end{align*}

\noindent while for $t,s\in[0,2\pi]$ such that $t+s<0,$ we have

\begin{align*}
&\left|\frac{1}{1-e^{-2\pi i\mathcal{M}_0(\xi_n)}}\exp\left(-\int_{t-s}^{t}i\mathcal{M}(r,\xi_n)dr\right)\widehat{f}(t-s,\xi_n)\right|\leqslant \\[2mm]
& \quad \exp\left(\int_{t-s}^{t}\Im\mathcal{M}(r,\xi_n)dr+\int_{t_n}^{t-s+2\pi}\Im\mathcal{M}(r,\xi_n)dr\right) = \\[2mm]
& \quad \exp\left(\int_{t_n}^{t}\Im\mathcal{M}(r,\xi_n)dr+2\pi\Im\mathcal{M}_0(\xi_n)\right).
\end{align*}

Since $\Im\mathcal{M}_0(\xi_n)\rightarrow0,$ by \eqref{est5}, the estimates above imply that
$|\widehat{u}(t,\xi_n)|\leqslant 4\pi,$ for all $t\in\mathbb{T}^1$ and for $n$ sufficiently large.
Hence, $\widehat{u}(\cdot,\xi_n)$ increases slowly and
\begin{equation*}
u(t,x)=\sum_{n=1}^{\infty}\widehat{u}(t,\xi_n) e^{ix\xi_n} \in \mathcal{D}'(\mathbb{T}^1\times\mathbb{T}^N).
\end{equation*}

If $t_0>\sup I,$ then $t_n>\sup I,$ for all $n$ sufficiently large, and \[|\widehat{u}(t_n,\xi_n)|=\int_{t_n-\inf I}^{t_n-\sup I}\phi(t_n-s)ds=\int_{0}^{2\pi}\phi(t)dt>0,\] on the other hand, if $t_0<\sup I,$ then $t_n<\inf I,$ for all $n$ sufficiently large, and \begin{eqnarray*}
|\widehat{u}(t_n,\xi_n)| &= & \left|\int_{t_n-\sup I+2\pi}^{t_n-\inf I+2\pi} \exp\left(-\int_{t_n-s}^{t_n}i\mathcal{M}(r,\xi_n)dr\right)\right. \\[2mm]
 & & \ \times \left. \exp\left(-\int_{t_n}^{t_n-s+2\pi}i\mathcal{M}(r,\xi_n)dr\right) \phi(t_n-s+2\pi)ds\right| \\[2mm]
 &=& e^{2\pi\Im\mathcal{M}_0(\xi_n)}\int_{0}^{2\pi}\phi(s)ds \ > \ (1/2)\int_{0}^{2\pi}\phi(s)ds  \ > \ 0,
\end{eqnarray*}
which implies that $\widehat{u}(\cdot,\xi_n)$ does not decay rapidly.

Hence $u\in\mathcal{D}'(\mathbb{T}^1\times\mathbb{T}^N)\setminus C^{\infty}(\mathbb{T}^1\times\mathbb{T}^N),$ and since $iLu=f$ (by Lemma \ref{lt1}), it follows that $L$ is not globally hypoelliptic.
\end{proof}

\subsection{Sufficient conditions}

We now present sufficient conditions to the global hypoellipticity of $L$.

\begin{theorem}\label{gt1}If the operator $L_0$ given by (\ref{MOCC}) is globally hypoelliptic and the function $\Im\mathcal{M}(t,\xi) = a(t)\beta(\xi) + b(t)\alpha(\xi) $ does not change sign, for sufficiently large $|\xi|$, then the operator $L$  given by \eqref{MO} is globally hypoelliptic.
\end{theorem}

\begin{proof}
Let $u \in \mathcal{D}'(\mathbb{T}^1\times\mathbb{T}^N)$ be a distribution such that $iLu=f,$ with $f\in C^\infty(\mathbb{T}^1\times\mathbb{T}^N).$ We will show that $u\in C^\infty(\mathbb{T}^1\times\mathbb{T}^N).$

By using partial Fourier series in the variable $x,$ it follows that $iLu=f$ if and only if
\begin{equation}\label{equa2}
\partial_t\widehat{u}(t,\xi)+i\mathcal{M}(t,\xi)\widehat{u}(t,\xi)=\widehat{f}(t,\xi),
\end{equation}
for all $t\in\mathbb{T}^1$ and for all $\xi\in\mathbb{Z}^N.$

Lemma \ref{lt1} implies that $\widehat{u}(\cdot,\xi) \in C^\infty(\mathbb{T}^1),$ for each $\xi\in\mathbb{Z}^N.$ Moreover, since $Z_{{\mathcal{M}}}$ is finite (thanks to Remark \ref{remarkcc}),
for $|\xi|$ sufficiently large the equation \eqref{equa2} has a unique solution, which may be written in the form \eqref{Solu-1} or \eqref{Solu-2}. Since $t\mapsto\Im{\mathcal{M}}(t,\xi)$ does not change sign for $|\xi|$ large enough, we conveniently write
\[
\widehat{u}(t,\xi)=\frac{1}{1-e^{-2\pi i\mathcal{M}_0(\xi)}}\int_{0}^{2\pi}\exp\left(-i\int_{t-s}^{t}\mathcal{M}(r,\xi)dr\right)\widehat{f}(t-s,\xi)ds,
\]
if $\xi$ is such that $\Im\mathcal{M}(t,\xi)\leqslant 0,$ for all $t\in\mathbb{T}^1,$ and
\[
\widehat{u}(t,\xi)= \frac{1}{e^{2\pi i\mathcal{M}_0(\xi)}-1}\int_{0}^{2\pi}\exp\left(i\int_{t}^{t+s}\mathcal{M}(r,\xi)dr\right)\widehat{f}(t+s,\xi)ds,
\]
if $\xi$ is such that $\Im\mathcal{M}(t,\xi)\geqslant 0,$ for all $t\in\mathbb{T}^1.$

Since $L_0$ is globally hypoelliptic, then there exist positive constants $C,$ $M,$ and $R,$ so that \begin{equation}\label{eqqq2}|1-e^{\pm 2\pi i\mathcal{M}_0(\xi)}|\geqslant C|\xi|^{-M},\end{equation} for all $|\xi|\geqslant R$ (Proposition \ref{lt2}).

Hence, for $|\xi|$ sufficiently large, the solution $\widehat{u}(\cdot,\xi)$ of \eqref{equa2} satisfies \[|\widehat{u}(t,\xi)|\leqslant \frac{2\pi}{C}|\xi|^{M}\|\widehat{f}(\cdot,\xi)\|_ {\infty}.\]

Similar estimates holds true for the derivatives $\partial_t^{n}\widehat{u}(t,\xi)$.

Thus, the rapid decaying of the sequence $\widehat{f}(\cdot,\xi)$ implies that $\widehat{u}(\cdot,\xi)$ decays rapidly.

Therefore, $u \in C^\infty(\mathbb{T}^1\times\mathbb{T}^N);$ consequently, $L$ is globally hypoelliptic.

\end{proof}

In the next sections we will see situations where $L$ is globally hypoelliptic, but there exist infinitely many indexes $\xi$ such that $\Im\mathcal{M}(t, \, \xi))$ changes sign. That is, the assumption that $\Im\mathcal{M}(t,\xi)$ does not change sign  is not necessary for the global hypoellipticity of the operator $L.$

\section{Logarithmic growth}\label{sectmlg}

From now on, the speed in which the symbol $p(\xi)$ goes to infinity will play a crucial point in the study of the global hypoellipticity of
\begin{equation*}
L=D_t+(a+ib)(t)P(D_x),\quad(t,x)\in\mathbb{T}^1\times\mathbb{T}^N.
\end{equation*}

We recall that $p(\xi)=\alpha(\xi)+i\beta(\xi),$ where both $\alpha(\xi)$ and $\beta(\xi)$ are real-valued functions in $S^m(\mathbb{Z}^N).$ In particular
\begin{equation}\label{tscest-homo-alphabeta}
|\alpha(\xi)| \leqslant C |\xi|^{m} \ \textrm{ and } \  |\beta(\xi)| \leqslant C |\xi|^{m}, \ \mbox{as }  |\xi| \to \infty.
\end{equation}

In this section our goal is to deal with the case where either $\alpha(\xi)$ or $\beta(\xi)$ has at most logarithmic growth.

\begin{definition}\label{defslg} A function $r:\mathbb{Z}^N\rightarrow \mathbb{C}$ has at most logarithmic growth if
\begin{equation*}
r(\xi) = O(\log(|\xi|)), \ \textrm{ as } \ |\xi| \to \infty,
\end{equation*}
that is, there are positive constants $\kappa$ and $n_0$ such that
\begin{equation}\label{log-estimate-1}
|r(\xi)| \leqslant  \kappa \log(|\xi|), \ \textrm{for all} \ |\xi|\geqslant  n_0.
\end{equation}
When this condition fails, we will say that $r(\xi)$ \textit{has super-logarithmic growth}.
\end{definition}

When either $\alpha(\xi)$ or $\beta(\xi)$ has at most logarithmic growth, the global hypoellipticity of $L$ is completely characterized by the following:

\begin{theorem}\label{gt3} Let $p(\xi)=\alpha(\xi)+i\beta(\xi) \in S^{m}(\mathbb{Z}^N)$ be a symbol.
\begin{itemize}
\item[i)] If $\alpha(\xi) = O(\log(|\xi|))$ and $\beta(\xi) = O(\log(|\xi|)),$ then $L$ is globally hypoelliptic if and only if $L_0$ is globally hypoelliptic.

\item[ii)] If $\alpha(\xi) = O(\log(|\xi|))$ and $\beta(\xi)$ has super-logarithmic growth, then $L$ is globally hypoelliptic if and only if $L_0$ is globally hypoelliptic and $a(t)$ does not change sign.

\item[iii)] If $\alpha(\xi)$ has super-logarithmic growth and $\beta(\xi) = O(\log(|\xi|)),$ then $L$ is globally hypoelliptic if and only if $L_0$ is globally hypoelliptic and $b(t)$ does not change sign.
\end{itemize}
\end{theorem}

In the particular case where $\beta \equiv 0$ we have the following:

\begin{corollary} If the symbol $p(\xi)$ is real-valued, then the operator $L$  is globally hypoelliptic if and only if $L_0$ is globally hypoelliptic and either
\begin{enumerate}
  \item[$i)$] $p(\xi)=O(\log(|\xi|))$; or
  \item[$ii)$] $p(\xi)$ has super-logarithmic growth and $b(t)$ does not change sign.
\end{enumerate}
\end{corollary}

\begin{remark}
When $p(\xi)$ is  a real-valued symbol having at most logarithmic growth, item i) shows that the behaviour of the function $b(t)$ plays no role in the global hypoellipticity of pseudo-differential operators of type \eqref{MO}, what means that the famous condition $(P)$ of Nirenberg-Treves, see \cite{NT70} and \cite{NT70-2}, is neither necessary nor sufficient to guarantee global hypoellipticity.

 On the other hand, item ii) is according  to the known result for vector fields $L=D_t+(a+ib)(t)D_x$ on $\mathbb{T}^2$ studied by Hounie in  \cite{HOU79}. We recall that in  this case,  the condition $L_0$ globally hypoelliptic means that either $b_0\neq 0$ or $a_0$ is an irrational non-Liouville number.

\end{remark}

We split the proof of Theorem \ref{gt3} in two subsections. In Subsection \ref{ssectamlg} we prove item $i)$ by using an argument of reduction to normal forms. The proof of items $ii)$ and $iii)$ are treated in Subsection \ref{ssectsuplg}, where the change of sign of the coefficients play an important role.

In Subsection \ref{querover} we show that the techniques developed in previous subsections can be applied to study a particular case where the symbol has super logarithmic growth.

Before proceeding with the proofs, we present two examples which illustrate that the condition $\Im\mathcal{M}(t,\xi)$ does not change sign in Theorem \ref{gt1} is not necessary for the global hypoellipticity of $L.$

\begin{example}\label{exampcc}
If $P(D_x)=(-\Delta_x)^{m/2}$ on $\mathbb{T}^N,$ with $m < 0,$ then by item $i)$ of the Theorem \ref{gt3} the operator
\begin{equation*}
L=D_t+[1+i\sin(t)](-\Delta_x)^{m/2}
\end{equation*}
is globally hypoelliptic since
\begin{equation*}
L_0=D_t+(-\Delta_x)^{m/2},
\end{equation*}
is globally hypoelliptic by Theorem \ref{t-2.1}. Notice that $\Im\mathcal{M}(t,\xi)=\sin(t)|\xi|^{m}$ changes sign for all $|\xi|>0.$
\end{example}

\begin{example}\label{exampcc2}
Assume that $P(D_x)=(-\Delta_x)^{m/2}+i(-\Delta_x)^{n/2}$ on $\mathbb{T}^N,$ where $m\leqslant 0$ and $n>0.$ Theorem \ref{gt3} item $iii)$ implies that
the operator
\begin{equation*}
L=D_t+[1+\cos(t)-i][(-\Delta_x)^{m/2}+i(-\Delta_x)^{n/2}]
\end{equation*}
is globally hypoelliptic, since $a(t)=1+\cos(t)\geqslant 0$ and
\begin{equation*}
L_0=D_t+(1-i)[(-\Delta_x)^{m/2}+i(-\Delta_x)^{n/2}]
\end{equation*}
is globally hypoelliptic. Indeed, the assumptions $m\leqslant 0$ and $n>0$ implies that, for $(\tau,\xi)\in\mathbb{Z}\times\mathbb{Z}_\ast^N$ such that $|\xi|>2^{1/n},$ we have
\begin{equation*}
|\tau+(1-i)(|\xi|^m+i|\xi|^n)|\geqslant||\xi|^n-|\xi|^m|\geqslant 1.
\end{equation*}

Hence, $L_0$ is globally hypoelliptic by Section \ref{sectcc}.

Notice that $\Im\mathcal{M}(t,\xi)=(1+\cos(t))|\xi|^n-|\xi|^m$ changes sign for infinitely many indexes, since $m\leqslant 0$ and $n>0.$
\end{example}

\subsection{Reduction to normal form}\label{ssectamlg}

\medskip

In this subsection we show that, under the assumption of growth at most logarithm
of the symbol, the study of the global hypoellipticity of $L$ and $L_0$ are equivalent.

In this situation we have
\begin{equation*}
\alpha(\xi)=O(\log(|\xi|)) \ \textrm{ and } \ \beta(\xi)=O(\log(|\xi|)), \mbox{ as } |\xi| \to \infty,
\end{equation*}
and the proof of item $i)$ of Theorem \ref{gt3} follows from Corollary \ref{reduct2-cor} bellow.

We introduce the following (formal) operators: for each distribution $u \in \mathcal{D}'(\mathbb{T}^1\times\mathbb{T}^N),$ we set
\begin{equation*}
\Psi_a(u) = \sum_{\xi \in \mathbb{Z}^N} e^{-i (A(t) - a_0 t) p(\xi)}\widehat{u}(t, \xi)e^{i x \xi},
\end{equation*}
and
\begin{equation*}
\Psi_b(u) = \sum_{\xi \in \mathbb{Z}^N} e^{(B(t) - b_0 t) p(\xi)}\widehat{u}(t, \xi)e^{i x \xi},
\end{equation*}
where
\begin{equation*}
A(t) = \int_{0}^{t}a(s) ds \ \textrm{ and } \ B(t) = \int_{0}^{t}b(s) ds.
\end{equation*}

\begin{proposition}\label{reduct2}If $\beta(\xi)=O(\log(|\xi|)),$ then  $\Psi_a$  is an isomorphism which satisfies \begin{equation}\label{Psi-a-1}
\Psi_a^{-1} \circ L \circ \Psi_a = L_{a_0}\end{equation} on both the spaces $\mathcal{D}'(\mathbb{T}^1\times\mathbb{T}^N)$ and $C^{\infty}(\mathbb{T}^1\times\mathbb{T}^N),$ where
\begin{equation*}
L_{a_0} \doteq D_t + (a_0 + i b(t)) P(D_x).
\end{equation*}

Analogously, if $\alpha(\xi)=O(\log(|\xi|)),$ then $\Psi_b$  is an isomorphism which satisfies \begin{equation}\label{Psi-b-1}
\Psi_b^{-1} \circ L \circ \Psi_b = L_{b_0}\end{equation} on both the spaces $\mathcal{D}'(\mathbb{T}^1\times\mathbb{T}^N)$ and $C^{\infty}(\mathbb{T}^1\times\mathbb{T}^N),$ where
\begin{equation*}
L_{b_0} \doteq D_t + (a(t) + i b_0)P(D_x).
\end{equation*}
\end{proposition}

The proof of this proposition consists in to show that $\Psi_a$ and $\Psi_b$  are well defined operators, in this case they are evidently linear operators with inverse
\begin{equation*}
\Psi_a^{-1}(v) = \sum_{\xi \in \mathbb{Z}^N} e^{i (A(t) - a_0 t) p(\xi)}\widehat{v}(t, \xi)e^{i x \xi},
\end{equation*}
and
\begin{equation*}
\Psi_b^{-1}(v) = \sum_{\xi \in \mathbb{Z}^N} e^{-(B(t) - b_0 t) p(\xi)}\widehat{v}(t, \xi)e^{i x \xi},
\end{equation*}
respectively, on both the spaces $\mathcal{D}'(\mathbb{T}^1\times\mathbb{T}^N)$ and $C^{\infty}(\mathbb{T}^1\times\mathbb{T}^N).$ Moreover,
identities \eqref{Psi-a-1} and \eqref{Psi-b-1} are easily verified.

Before starting this proof, let us state the reduction to the normal form:

\begin{corollary}\label{reduct2-cor}
If $\beta(\xi)=O(\log(|\xi|))$ (respectively $\alpha(\xi)=O(\log(|\xi|))$), then $L$ is globally hypoelliptic if and only if $L_{a_0}$ (respectively $L_{b_0}$) is globally hypoelliptic.
\end{corollary}

\begin{proof}
The validity of the identity $L_{a_0}=\Psi_a^{-1}\circ L \circ\Psi_a$ on both $\mathcal{D}'(\mathbb{T}^1\times\mathbb{T}^N)$ and $C^{\infty}(\mathbb{T}^1\times\mathbb{T}^N)$ imply that $L$ is globally hypoelliptic if and only if $L_{a_0}$ is globally hypoelliptic.

In fact, assume that $L$ is globally hypoelliptic and let $u \in \mathcal{D}'(\mathbb{T}^1\times\mathbb{T}^N)$ such that $L_{a_0} u = f \in C^{\infty}(\mathbb{T}^1\times\mathbb{T}^N)$. Since $v = \Psi_{a}(u)\in \mathcal{D}'(\mathbb{T}^1\times\mathbb{T}^N) $ satisfy $Lv= \Psi_{a}(f)\in C^{\infty}(\mathbb{T}^1\times\mathbb{T}^N),$ it follows that $v\in C^{\infty}(\mathbb{T}^1\times\mathbb{T}^N),$ since $L$ is globally hypoelliptic.

Hence, $u=\Psi_{a}^{-1}(v)\in C^{\infty}(\mathbb{T}^1\times\mathbb{T}^N),$ which implies that $L_{a_0}$ is globally hypoelliptic. The converse is similar.

Analogously, the validity of the identity $L_{b_0}=\Psi_b^{-1}\circ L \circ\Psi_b$ on both $\mathcal{D}'(\mathbb{T}^1\times\mathbb{T}^N)$ and $C^{\infty}(\mathbb{T}^1\times\mathbb{T}^N)$ will imply that $L$ is globally hypoelliptic if and only if $L_{b_0}$ is globally hypoelliptic.

\end{proof}

The following estimates will be useful in the proof of Proposition \ref{reduct2}.

\begin{lemma}\label{prop-exp-ine}
Consider $p\in S^{m}(\mathbb{Z}^N).$ Given $k \in \mathbb{N}_0,$ there are positive constants $C$  and $n_0$ such that
\begin{equation*}
|\partial_t^k (e^{-i (A(t) - a_0 t) p(\xi)})| \leqslant  C |\xi|^{k m}e^{\beta(\xi)(-A(t) + a_0 t)},
\end{equation*}
and
\begin{equation*}
|\partial_t^k (e^{(B(t) - b_0 t) p(\xi)})| \leqslant  C |\xi|^{k m} e^{\alpha(\xi)(B(t) - b_0 t)},
\end{equation*}
for each $|\xi|\geqslant  n_0$.
\end{lemma}

\begin{proof} For $k=0$ these estimates are evident. If the first estimate holds for $\ell \in \{0, 1, \ldots, k\},$ then we have:

\begin{align*}
|\partial_t^{k+1} (e^{-i (A(t) - a_0 t) p(\xi)})| & \leqslant  | p(\xi)| \sum_{\ell = 0}^{k}\binom{k}{\ell}|\partial_t^{\ell}e^{-i (A(t) - a_0 t) p(\xi)}|\\
                                                  & \quad \times \sup_{t \in \mathbb{T}^1}|\partial_t^{k-\ell}(a(t) - a_0 t)| \\
                                                  & \leqslant  C | p(\xi)| \sum_{\ell = 0}^{k}\binom{k}{\ell}|\partial_t^{\ell}e^{-i (A(t) - a_0 t) p(\xi)}| \\
                                                  & \leqslant  C | p(\xi)| e^{\beta(\xi)(-A(t) + a_0 t)} \sum_{\ell = 0}^{k}\binom{k}{\ell} |\xi|^{\ell m} \\
                                                  & \leqslant          C |\xi|^{(k+1) m} e^{\beta(\xi)(-A(t) + a_0 t)},
\end{align*}
where we are using $|p(\xi)| \leqslant  C |\xi|^m$, as  $|\xi| \to \infty$.

The second estimate can be obtained by using similar arguments.

\end{proof}

\begin{proof}[{Proof of Proposition \ref{reduct2}}]
We have to verify only that $\Psi_a$ and $\Psi_b$ are well defined linear operators on both $\mathcal{D}'(\mathbb{T}^1\times\mathbb{T}^N)$ and $C^{\infty}(\mathbb{T}^1\times\mathbb{T}^N)$ whenever
\begin{equation*}
\beta(\xi)=O(\log(|\xi|)) \ \textrm{ and } \ \alpha(\xi)=O(\log(|\xi|)),
\end{equation*}
respectively.

Fixed $u \in \mathcal{D}'(\mathbb{T}^1\times\mathbb{T}^N),$ we must study the behavior of the Fourier coefficients
\begin{equation*}
\psi_a(t, \xi) = e^{-i (A(t) - a_0 t) p(\xi)}\widehat{u}(t, \xi), \ \textrm{for all} \ \xi \in \mathbb{Z}^N,
\end{equation*}
and
\begin{equation*}
\psi_b(t, \xi) = e^{(B(t) - b_0 t) p(\xi)}\widehat{u}(t, \xi), \ \textrm{for all} \ \xi \in \mathbb{Z}^N.
\end{equation*}

Given $u \in \mathcal{D}'(\mathbb{T}^1\times\mathbb{T}^N),$ it follows by Lemma \ref{prop-exp-ine} the existence of positive constants $C$ and $M$ such that
\begin{equation}\label{psi-a-1}
|\langle\psi_a(t,\xi),\phi\rangle| \leqslant  C |\xi|^{M}\|\phi\|_{M}\sup_{t\in\mathbb{T}^1}|e^{\beta(\xi)(-A(t) + a_0 t)}|,
\end{equation}
and
\begin{equation}\label{psi-b-2}
|\langle\psi_b(t,\xi),\phi\rangle| \leqslant  C |\xi|^{M}\|\phi\|_{M}\sup_{t\in\mathbb{T}^1}|e^{\alpha(\xi)(B(t) - b_0 t)}|,
\end{equation}
for $|\xi|$ large enough, where
\begin{equation*}
\|\phi\|_{M} \doteq \max\{|\partial^{\alpha}  \phi(t)|; \ \alpha \leqslant M, \ t\in\mathbb{T}^1\}.
\end{equation*}

If $\beta(\xi)=O(\log(|\xi|)),$ by \eqref{log-estimate-1} there exist $\kappa>0$ and $n_0 \in \mathbb{N}$ such that
\begin{equation}\label{log-estimate-3}
|\beta(\xi)| \leqslant  \log(|\xi|^{\kappa}), \ \textrm{for all} \ |\xi|\geqslant  n_0.
\end{equation}

Now, take $\delta_1<0$ and $\delta_2>0$ satisfying
\begin{equation}\label{log-estimate-3*}
\delta_1 \leqslant  -A(t) + a_0 t \leqslant  \delta_2, \ \textrm{for all} \ t \in \mathbb{T}^1.
\end{equation}

The inequalities \eqref{log-estimate-3} and \eqref{log-estimate-3*} imply that, for all $|\xi|\geqslant  n_0$ we have:
\begin{equation*}
\beta(\xi) (-A(t) + a_0 t)) \leqslant   \left \{
\begin{array}{l}
\log(|\xi|^{\kappa\delta_2}), \ \textrm{ if } \  \beta(\xi) >0, \\[3mm]
\log(|\xi|^{-\kappa\delta_1}), \ \textrm{ if } \  \beta(\xi) <0.
\end{array}\right.
\end{equation*}
Hence,
\begin{equation}\label{ne1}
e^{\beta(\xi)(-A(t) + a_0 t)} \leqslant  |\xi|^{\delta_3}, \ \textrm{ as } \ |\xi| \to \infty,
\end{equation}
where $\delta_3 = \max\{\kappa\delta_2 , -\kappa\delta_1\}$.

With similar ideas, by using the fact that $\alpha(\xi)=O(\log(|\xi|)),$ we obtain $\delta_4>0$ such that
\begin{equation}\label{ne2}
e^{\alpha(\xi)(B(t) - b_0 t)} \leqslant  |\xi|^{\delta_4}, \ \textrm{ as } \ |\xi| \to \infty.
\end{equation}

Then, by \eqref{psi-a-1}, \eqref{psi-b-2} and the last two inequalities
\begin{equation*}
|\langle\psi_a(t,\xi),\phi\rangle| \leqslant  C |\xi|^{M +\delta_3}\|\phi\|_{M}
\end{equation*}
and
\begin{equation*}
|\langle\psi_b(t,\xi),\phi\rangle| \leqslant  C |\xi|^{M + \delta_4}\|\phi\|_{M},
\end{equation*}
for all $|\xi|$ sufficiently large and for all $\phi\in C^{\infty}(\mathbb{T}^1);$ thus $\Psi_a \cdot u \in \mathcal{D}'(\mathbb{T}^1\times\mathbb{T}^N)$ and $\Psi_b \cdot u \in \mathcal{D}'(\mathbb{T}^1\times\mathbb{T}^N).$

Finally, if $u \in C^{\infty}(\mathbb{T}^1\times\mathbb{T}^N),$ then Lemma \ref{prop-exp-ine} and the rapid decaying of $\widehat{u}(\cdot,\xi)$ imply that for each $k \in \mathbb{N}_0$ we obtain $C_k>0$ and $M_k \in \mathbb{R}$ such that
\begin{equation*}
|\partial_t^k \psi_a(t, \xi)| \leqslant  C_k |\xi|^{M_k}e^{\beta(\xi)(-A(t) + a_0 t)}\sum_{j=0}^{k}|\partial_t^{k}\widehat{u}(t,\xi)|
\end{equation*}
and
\begin{equation*}
|\partial_t^k \psi_b(t, \xi)| \leqslant  C_k |\xi|^{M_k} e^{\alpha(\xi)(B(t) - b_0 t)}\sum_{j=0}^{k}|\partial_t^{k}\widehat{u}(t,\xi)|,
\end{equation*}
for $|\xi|$ large enough.

By using again \eqref{ne1} and \eqref{ne2}, and from the rapid decaying of $\widehat{u}(\cdot,\xi),$ it follows that $\Psi_a(u)$ and $\Psi_b(u)$ are in $C^{\infty}(\mathbb{T}^1 \times \mathbb{T}^N)$, what finishes the proof of Proposition \ref{reduct2}.

\hfill $\square$
\end{proof}

\subsection{Change of sign}\label{ssectsuplg}

Our focus now is to prove item $ii)$ of Theorem \ref{gt3}, in which $\alpha(\xi)$ has at most logarithmic growth and $\beta(\xi)$ has super-logarithmic growth. Notice that, in this case, the global hypoellipticity of $L$ cannot be reduced to the global hypoellipticity of a constant coefficient operator.

The proof of item $iii)$ of Theorem \ref{gt3} consists in slight modifications of the techniques used in the proof of item $ii).$ Since the argument is quite similar, it will be omitted.

\begin{proof}[Proof of item $ii)$ of Theorem \ref{gt3}]
We recall that the hypothesis in this case are $\beta(\xi)$ has super-logarithmic growth and $\alpha(\xi)=O(\log(|\xi|))$, hence, in view of Corollary \ref{reduct2-cor} we may assume that $b(t)$ is constant, $b\equiv b_0.$

\bigskip
\textit{Sufficiency:}

Assume that $a(t)$ does not change sign and that $L_0$ is globally hypoelliptic.

Let $u\in\mathcal{D}' (\mathbb{T}^1\times\mathbb{T}^N)$ be such that $iLu=f\in C^{\infty}(\mathbb{T}^1\times\mathbb{T}^N).$ We will show that $u \in C^{\infty}(\mathbb{T}^1\times\mathbb{T}^N).$ By the Fourier series in the variable $x,$ we are led to the equations
\begin{align}\label{equa3}
&\widehat{f}(t,\xi)= \, \partial_t\widehat{u}(t,\xi)+i\mathcal{M}(t,\xi)\widehat{u}(t,\xi)\\ \nonumber
&=  \, \partial_t\widehat{u}(t,\xi)+\!\big[\!-(b_0\alpha(\xi)+a(t)\beta(\xi))+i(a(t)\alpha(\xi)-b_0\beta(\xi))\big]\widehat{u}(t,\xi),
\end{align}
for all $t\in\mathbb{T}^1$ and for all $\xi\in\mathbb{Z}^N.$

By applying Lemma \ref{lt1} to equation \eqref{equa3}, it follows that $t \in \mathbb{T}^1 \mapsto \widehat{u}(t,\xi)$ is smooth, for each $\xi\in\mathbb{Z}^N.$ Since $Z_{{\mathcal{M}}}$ is finite (Remark \ref{remarkcc}), for $|\xi|$ sufficiently large, equation \eqref{equa3} has a unique solution. This solution can be written by
\[\widehat{u}(t,\xi)=\frac{1}{1-e^{-2\pi i\mathcal{M}_0(\xi)}}\int_{0}^{2\pi} \exp\left(-\int_{t-s}^{t} i\mathcal{M}(r,\xi)dr\right) \widehat{f}(t-s,\xi)ds,\]
if $\xi$ is such that $a(t)\beta(\xi)\leqslant 0,$ for all $t\in\mathbb{T}^1,$ and
\[\widehat{u}(t,\xi)=\frac{1}{e^{2\pi i\mathcal{M}_0(\xi)}-1}\int_{0}^{2\pi}\exp\left(\int_{t}^{t+s} i\mathcal{M}(r,\xi)dr\right) \widehat{f}(t+s,\xi)ds,\]
if $\xi$ is such that $a(t)\beta(\xi)\geqslant 0,$ for all $t\in\mathbb{T}^1.$

Since $\alpha(\xi)=O(\log(|\xi|)),$ there exists $K>0$ such that
\[e^{\alpha(\xi)sb_0}\leqslant |\xi|^{K|b_{0}|},\]
for $|\xi|$ sufficiently large and $s\in[0,2\pi].$ Thus, for $|\xi|$ large enough and such that $a(t)\beta(\xi)\leqslant 0,$ for all $t\in\mathbb{T}^1,$ we have
\begin{align*}
\left|\exp\left(-\int_{t-s}^{t}i\mathcal{M}(r,\xi)dr\right)\right| & = \ \exp\left( \alpha(\xi)sb_0+\int_{t-s}^{t}a(r)\beta(\xi)dr \right) \\[2mm]
 & \ \leqslant e^{\alpha(\xi)sb_0}\leqslant |\xi|^{K|b_{0}|}.
\end{align*}

Similarly, for $|\xi|$ large enough and such that $a(t)\beta(\xi)\geqslant 0,$ for all $t\in\mathbb{T}^1,$ we have
\begin{align*}
\left|\exp\left(\int_{t}^{t+s}i\mathcal{M}(r,\xi)dr \right)\right| & = \ \exp\left( -\alpha(\xi)sb_0-\int_{t}^{t+s}a(r)\beta(\xi)dr \right) \\[2mm]
 & \ \leqslant e^{-\alpha(\xi)sb_0}\leqslant |\xi|^{K|b_{0}|}.
\end{align*}

Finally, as in the proof of Theorem \ref{gt1}, the global hypoellipticity of $L_0$ give us a control as in \eqref{eqqq2}, and the rapid decaying of $\widehat{f}(\cdot,\xi)$ imply that $\widehat{u}(\cdot,\xi)$ decays rapidly. Hence, $u$ belongs to $C^{\infty}(\mathbb{T}^1\times\mathbb{T}^N)$ and $L$ is globally hypoelliptic.

\bigskip

\textit{Necessity:}

By Theorem \ref{ncm2}, it is enough to prove that the changing of sign of $a(t)$ implies that $L$ is not globally hypoelliptic.

We will exhibit a smooth function
$$f(t,x)=\sum_{n=1}^{\infty}\widehat{f}(t,\xi_n)e^{ix\xi_n},$$
for which $iLu=f$ has a solution in  $\mathcal{D}'(\mathbb{T}^1\times\mathbb{T}^N)\setminus C^{\infty}(\mathbb{T}^1\times\mathbb{T}^N).$

Our assumptions on $\beta(\xi)$ imply that we may choose a sequence $\{\xi_n\},$ such that $|\xi_n|$ is strictly increasing, $|\xi_n|\geqslant n,$ and $|\beta(\xi_n)|\geqslant \log(|\xi_n|^n),$ for all $n\in\mathbb{N}.$

By passing to a subsequence we may assume that
either  $\beta(\xi_n)>0,$ for all $n,$ or $\beta(\xi_n)<0,$ for all $n.$

Without loss of generality, we also may assume that
$$b_0\alpha(\xi_n)+a_0\beta(\xi_n)\leqslant 0,$$
for all $n\in\mathbb{N}.$ Indeed, in the other case, it is enough to consider $-L$ and to change the variable $t$ by $-t.$

Suppose we are in the case  $\beta(\xi_n)>0$ for all $n$ (the other case is similar).

Set
$$M_a=\displaystyle\max_{0\leqslant t,s\leqslant 2\pi}\int_{t-s}^{t}a(r)dr=\int_{t_0-s_0}^{t_0}a(r)dr.$$

Since $a(t)$ changes sign, $M_a>0$ and $s_0\in(0,2\pi);$ moreover, without loss of generality (by performing a translation in the variable $t$) we may assume that $t_0$ and $\sigma_0\doteq t_0-s_0$ belong to the open interval $(0,2\pi).$

Let $\phi\in C^{\infty}_ {c}((\sigma_0-\epsilon,\sigma_0+\epsilon))$ be a function such that $0\leqslant \phi(t)\leqslant 1,$ and $\phi(t)=1$ in a neighborhood of $[\sigma_0-\epsilon/2,\sigma_0+\epsilon/2].$

We then define $\widehat{f}(\cdot,\xi_n)$ by the $2\pi-$periodic extension of the function
\[(1-e^{-2\pi i\mathcal{M}_0(\xi_n)})\phi(t)\exp\left(i\int_{t}^{t_0}\Re{M}(r,\xi_n)dr \right)
e^{-\beta(\xi_n)M_a} e^{\alpha(\xi_n)(t-t_0)b_0}.\]

Since $b_0\alpha(\xi_n)+a_0\beta(\xi_n)\leqslant 0$, we have that $1-e^{-2\pi i\mathcal{M}_0(\xi_n)}$ is bounded and for $t\in[0,2\pi]$ we have
$$e^{\alpha(\xi_n)(t-t_0)b_0}\leqslant e^{|\alpha(\xi_n)|2\pi|b_0|},$$
which increases slowly, since $\alpha(\xi)=O(\log(|\xi|)).$

Moreover, by using estimate \eqref{tscest-homo-alphabeta}, the term $e^{-\beta(\xi_n)M_a}$ will imply that $\widehat{f}(\cdot,\xi_n)$ decays rapidly, since $\beta(\xi_n)>\log(|\xi_n|^n).$

By Proposition \ref{tt1} we may assume that $Z_{\mathcal{M}}$ is finite and, by passing to a subsequence, that $1-e^{-2\pi i\mathcal{M}_0(\xi_n)}\neq0,$ then we define
\[\widehat{u}(t,\xi_n)=\frac{1}{1-e^{-2\pi i\mathcal{M}_0(\xi_n)}}\int_{0}^{2\pi}\exp\left( -\int_{t-s}^{t}i\mathcal{M}(r,\xi_n)dr \right) \widehat{f}(t-s,\xi_n) ds.
\]

For all $s,t\in[0,2\pi],$ we have
\[
\left|\frac{1}{1-e^{-2\pi i\mathcal{M}_0(\xi_n)}}\widehat{f}(t-s,\xi_n) \right|\leqslant e^{|\alpha(\xi_n)|4\pi|b_0|}e^{-\beta(\xi_n)M_a}.
\]
Thus,
\begin{align*}
|\widehat{u}(t,\xi_n)| &\leqslant \, \int_{0}^{2\pi}\!\!\exp\left(-\beta(\xi_n)\Big(M_a-\int_{t-s}^{t}a(r)dr\Big)\right)e^{\alpha(\xi_n)sb_0}e^{|\alpha(\xi_n)|4\pi|b_0|}ds \\[2mm]
                       &\leqslant \, 2\pi e^{|\alpha(\xi_n)|6\pi|b_0|}.
\end{align*}
 This estimate imply that the sequence $\widehat{u}(\cdot,\xi_n)$ increases slowly,
since $\alpha(\xi)=O(\log(|\xi|)).$ Hence
\begin{equation*}
u=\sum_{n=1}^{\infty}\widehat{u}(t,\xi_n) e^{ix\xi_n} \in \mathcal{D}'(\mathbb{T}^1\times\mathbb{T}^N).
\end{equation*}

Note that
\[
|\widehat{u}(t_0,\xi_n)| \geqslant \int_{s_0-\delta/2}^{s_0+\delta/2}\exp\left(-\beta(\xi_n)\Big(M_a-\int_{t_0-s}^{t_0}a(r)dr \Big)\right)ds.
\]

Since $M_a-\int_{t_0-s}^{t_0}a(r)dr\geqslant 0$ and $s_0$ is a zero of order even, the Laplace Method for Integrals implies that
\[|\widehat{u}(t_0,\xi_n)|\geqslant C(\beta(\xi_n))^{-1/2}\geqslant C(1+|\xi_n|^2)^{-m/4}\geqslant C2^{-m/4}|\xi_n|^{-m/2},\] where $C$ and $m$ are positive constants and do not depend on $n.$ This estimate implies that
$\widehat{u}(\cdot,\xi_n)$ does not decay rapidly.

Hence, $u \in \mathcal{D}'(\mathbb{T}^1\times\mathbb{T}^N)\setminus C^{\infty}(\mathbb{T}^1\times\mathbb{T}^N).$ Therefore, $L$ is not globally hypoelliptic, since $iLu=f$ by Lemma \ref{lt1}.

When $\beta(\xi_n)<0$ for all $n,$ we repeat the constructions above, where now we use
$$M_a=\displaystyle\min_{0\leqslant t,s\leqslant 2\pi}\int_{t-s}^{t}a(r)dr.$$

The proof of Theorem \ref{gt3} - item $ii)$ is complete.
\hfill $\square$ \vspace{2mm}
\end{proof}

\begin{remark}In the proof of sufficiency in Theorem \ref{gt3} item $ii),$ was not necessary suppose that $\beta(\xi)$ has super-logarithmic  growth. Moreover, we observe that this proof is not  consequence of Theorem \ref{gt1}, since $t\mapsto\Im\mathcal{M}(t,\xi)=a(t)\beta(\xi)+b(t)\alpha(\xi)$ may change sign, even when $a(t)$ does not change sign.
\end{remark}

We finish this subsection with an additional example which also exhibit a globally hypoelliptic operator in a situation in which $t\in\mathbb{T}^1\mapsto \Im\mathcal{M}(t,\xi)$ changes sign, for infinitely many indexes $\xi.$

\begin{example}\label{exampsign1} Let $b(t)$ be a $2\pi-$periodic extension of a real smooth nonzero function defined on $(0,2\pi)$ with integral equals to zero. Let $a(t)$ be the $2\pi-$periodic extension of the function $1-\psi,$ where $\psi\in C^\infty_c((0,2\pi),\mathbb{R}),$ $0\leqslant \psi(t)\leqslant 1$ and $\psi\equiv1$ in a neighborhood of support of $b(t).$

If $P(D_x)$ has symbol $p(\xi)=1+i(|\xi| \log(1 + |\xi|))$, then
\begin{equation*}
L=D_t +(a(t)+ib(t))P(D_x), \  (t,x)\in \mathbb{T}^1\times\mathbb{T}^N,
\end{equation*}
is globally hypoelliptic by Theorem \ref{gt3} - item $ii)$. Note that,
\begin{equation*}
t\mapsto \Im\mathcal{M}(t,\xi)=a(t)|\xi| \log(1 + |\xi|)+b(t),
\end{equation*}
changes sign for all indexes $\xi\in\mathbb{Z}^N.$
\end{example}

\subsection{A particular class of operators}\label{querover}

The aim of this subsection is to notice that there is a particular class of operators, which includes cases in which both $\alpha(\xi)$ and $\beta(\xi)$ have super-logarithmic growth, where the study of the global hypoellipticity follows from adaptations of the techniques used in the proof of Theorem \ref{gt3}.

For example, if $p(\xi)=\alpha(\xi)+i(1+\alpha(\xi))$ and $\alpha(\xi)$ has super-logarithmic growth, we cannot apply Theorem \ref{gt3} to study the global hypoellipticity of the operator
\begin{equation*}
D_t+(\cos^2(t)+i\sin(t))P(D_x),
\end{equation*}
but notice that $\Im\mathcal{M}(t,\xi)$ splits in the form
\begin{equation*}
[\sin(t)+\cos^2(t)]\alpha(\xi)+\cos^{2}(t).
\end{equation*}

Hence, $\Im\mathcal{M}(t,\xi)$ satisfies the assumptions concerning the speed of growth which was assumed in Theorem \ref{gt3}. We claim that the operator above is not globally hypoelliptic, since $[\sin(t)+\cos^2(t)]$ changes sign.

More generally, with similar arguments of those used in the proof of Theorem \ref{gt3}, we may give a complete answer about the global hypoellipticity of the operator $L,$ given by \eqref{MO}, in the case where $\Im\mathcal{M}(t,\xi)$ splits in the following way:
\begin{equation}\label{splitm}\Im\mathcal{M}(t,\xi)=\tilde{a}(t)\gamma(\xi)+\tilde{b}(t)\eta(\xi),\end{equation}
where $\tilde{a}(t)$ and $\tilde{b}(t)$ are real smooth functions on $\mathbb{T}^1,$ and $\gamma(\xi)$ and $\eta(\xi)$ are real valued toroidal symbols, such that either  $\gamma(\xi)=O(\log(|\xi|))$ or $\eta(\xi)=O(\log(|\xi|)).$

\begin{theorem}\label{conjsign2}Let $L$ be the operator defined in \eqref{MO} and assume that the decomposition \eqref{splitm} is true. Then $L$ is globally hypoelliptic if and only if $L_0$ is globally hypoelliptic and $\tilde{a}(t)$ (respectively $\tilde{b}(t)$) does not change sign whenever $\gamma(\xi)$ (respectively $\eta(\xi)$) has super-logarithmic growth.
\end{theorem}

Observe that, under the assumptions in Theorem \ref{conjsign2} and assuming that $\gamma(\xi)$ has super-logarithmic growth, the converse of
Theorem \ref{gt1} holds true provided that the function $\Im\mathcal{M}(t,\xi)=\tilde{a}(t)\gamma(\xi)+\tilde{b}(t)\alpha(\xi)$ changes sign if and only if $\tilde{a}$ changes sign.
However, as we saw in Example \ref{exampsign1}, this property does not hold in general.

Bellow we present other interesting examples in this direction.

\begin{example}\label{exampabncs} If $a(t)$ and $b(t)$ do not vanish identically and are $\mathbb{R}-$linearly dependent functions, we may write
\begin{equation*}
\Im\mathcal{M}(t,\xi)=b(t)(\alpha(\xi)+\lambda\beta(\xi)),
\end{equation*}
with $\lambda\in\mathbb{R}\setminus\{0\}.$ In this case, Theorem \ref{conjsign2} gives a complete answer about the global hypoellipticity of
$L.$

When $\alpha(\xi)+\lambda\beta(\xi)=O(\log(|\xi|)),$ $L$ is globally hypoelliptic if and only if $L_0$ is globally hypoelliptic.

For instance, if $a(t)=-b(t)$ and $\beta(\xi)=1+\alpha(\xi),$ then $\Im\mathcal{M}(t,\xi)=-a(t).$ Hence, $L$ is globally hypoelliptic even if $b(t)$ changes sign.

When $\alpha(\xi)+\lambda\beta(\xi)$ has super-logarithmic growth, $L$ is globally hypoelliptic if and only if $L_0$ is globally hypoelliptic  and $b(t)$       does not change sign.
\end{example}

\begin{example}\label{exabncs2}
When $a(t)$ and $b(t)$ are $\mathbb{R}-$linearly independent functions, $L$ may be not globally hypoelliptic even if both $a(t)$ and $b(t)$ do not change sign. Indeed, we may find non-zero integers $p$ and $q$ so that $a(t)p+b(t)q$ changes sign (see Lemma 3.1 of \cite{BDGK}).

If, for instance, $\alpha(\xi)=q\gamma(\xi)$ and $\beta(\xi)=p\gamma(\xi),$ in which $\gamma(\xi)$ has super-logarithmic growth, then Theorem \ref{conjsign2} implies that $L$ is not globally hypoelliptic.
\end{example}

\section{Super-logarithmic growth}\label{sectfr}

The purpose of this section is to present additional results about the global hypoellipticity of the operator $L$ given by \eqref{MO}, which we recall
\[L=D_t+(a+ib)(t)P(D_x),\quad (t,x)\in\mathbb{T}^1 \times \mathbb{T}^N,\]
where $a(t)$ and $b(t)$ are real smooth functions on $\mathbb{T}^1,$ and $P(D_x)$ is a pseudo-differential operator on $\mathbb{T}^N,$ with symbol $p(\xi)=\alpha(\xi)+i\beta(\xi), \ \xi\in\mathbb{Z}^N.$

We consider a more general situation where either $\alpha(\xi)$ or $\beta(\xi)$ has super-logarithmic growth and we present a necessary condition for the global hypoellipticity of $L,$ which is given by a control in the sign of certain functions.

Precisely, assume that $\beta(\xi)$ has super-logarithmic growth and let
\[E_{\alpha,\beta}\doteq\big\{K\in\mathbb{R};\,\,\mbox{there exists } \{\xi_n\}\subset\mathbb{Z}^N \ \mbox{satisfying }(\ast)\big\},\]
\begin{eqnarray*} (\ast) & & \left\{
\begin{array}{ll}
|\xi_n|\rightarrow\infty; \\[3mm]
\alpha(\xi_n)/\beta(\xi_n)\rightarrow K, \mbox{ as } n\rightarrow\infty;\\[3mm]
|\beta(\xi_n)|\geqslant n\log(|\xi_n|), \mbox{ for all } n\in\mathbb{N}.
\end{array}\right.
\end{eqnarray*}


In this case, we prove that $L$ is not globally hypoelliptic if there exists $K\in E_{\alpha,\beta}$ such that the function $t\in\mathbb{T}^1 \mapsto a(t)+b(t)K$ changes sign.

An analogous result holds when $\alpha(\xi)$ has super-logarithmic growth. In this case, $L$ is not globally hypoelliptic if there exists $C\in E_{\beta,\alpha}$ such that the function $t\in\mathbb{T}^1 \mapsto b(t)+a(t)C$ changes sign.

In particular, we obtain a necessary condition for the global hypoellipticity of $L$ when either $\alpha(\xi)$ or $\beta(\xi)$ has super-logarithmic growth and the limit $\lim_{|\xi|\rightarrow\infty}\alpha(\xi)/\beta(\xi)$ exists. When the order of growth of $\alpha(\xi)$ is faster (respectively slower) than the order of growth of $\beta(\xi),$ the operator $L$  is not globally hypoelliptic if $b(t)$ (respectively $a(t)$) changes sign (see Corollary \ref{cornew}).

\begin{theorem}\label{ar1} If $\beta(\xi)$ has super-logarithmic growth, $K\in E_{\alpha,\beta}$ and the function $t\in\mathbb{T}^1 \mapsto a(t)+b(t)K$ changes sign, then $L$ given by \eqref{MO} is not globally hypoelliptic. Similarly, if  $\alpha(\xi)$ has super-logarithmic growth, $C\in E_{\beta,\alpha},$ and $t\in\mathbb{T}^1 \mapsto b(t)+a(t)C$ changes sign, then $L$ is not globally hypoelliptic.
\end{theorem}

\begin{proof}We consider the situation in which $\beta(\xi)$ has super-logarithmic growth and $K\in E_{\alpha,\beta}.$ The other situation is analogous.

We assume that $t\in\mathbb{T}^1 \mapsto a(t)+b(t)K$ changes sign and prove that $L$ is not globally hypoelliptic.

The assumptions on $\beta$ and $K$ imply that there exists a sequence $\{\xi_n\}$ such that $|\xi_n|$ is strictly increasing, $|\xi_n|>n,$ $|\beta(\xi_n)|\geqslant n\log(|\xi_n|),$ $\alpha(\xi_n)/\beta(\xi_n)\rightarrow K,$ and $\xi_n\not\in Z_{\mathcal{M}},$ for all $n.$ Note that we are assuming that $Z_\mathcal{M}$ is finite, otherwise, by Proposition \ref{tt1}, there is nothing to prove.

Without loss of generality, suppose that $b_0\alpha(\xi_n)+a_0\beta(\xi_n)\leqslant 0,$ for all $n.$ Indeed, if necessary we can consider $-L$ and perform the change of variable $t$ by $-t.$

By using a subsequence, we may assume that either $\beta(\xi_n)<0,$ for all $n,$ or $\beta(\xi_n)>0,$ for all $n.$

Suppose first that $\beta(\xi_n)>0,$ for all $n.$ For each $n,$ set \[M_n=\max_{0\leqslant t,s\leqslant 2\pi}\left\{\int_{t-s}^{t}a(r)+b(r)\frac{\alpha(\xi_n)}{\beta(\xi_n)}dr\right\}=\int_{t_n-s_n}^{t_n}a(r)+b(r)\frac{\alpha(\xi_n)}{\beta(\xi_n)}dr.\]

Again, by passing to a subsequence, there exist $t_0$ and $s_0$ such that $t_n\rightarrow t_0$ and $s_n\rightarrow s_0,$ as $n\rightarrow\infty.$ Since $\alpha(\xi_n)/\beta(\xi_n)\rightarrow K,$ as $n\rightarrow\infty,$ it follows that
\[
  \int_{t_0-s_0}^{t_0}a(r)+b(r)K dr=\max_{0\leqslant t,s\leqslant 2\pi}\int_{t-s}^{t}a(r)+b(r)Kdr\doteq M_{ab}.
\]

Since $a(t)+b(t)K$ changes sign, we have $M_{ab}>0$ and $s_0\in(0,2\pi).$ Performing a translation in the variable $t,$ we may assume that $t_0,$ $s_0$ and $\sigma_0\doteq t_0-s_0$ belong to $(0,2\pi).$

Choose $\epsilon>0$ small enough so that $0<\sigma_0-\epsilon$ and $\sigma_0+\epsilon<t_0.$ Consider $\phi$ belonging to $C^{\infty}_{c}((\sigma_0-\epsilon,\sigma_0+\epsilon),\mathbb{R})$ such that $0\leqslant\phi(t)\leqslant 1$ and $\phi(t)=1$ for all $t\in[\sigma_0-\epsilon/2,\sigma_0+\epsilon/2].$

Finally, we define $\widehat{f}(\cdot,\xi_n)$ as being the $2\pi-$periodic extension of
\[
(1-e^{-2\pi i\mathcal{M}_0(\xi_n)})\phi(t)\exp\left(i\int_{t}^{t_n}\Re\mathcal{M}(r,\xi_n)dr \right)e^{-\beta(\xi_n)M_n}.
\]

Note that $1-e^{-2\pi i\mathcal{M}_0(\xi_n)}$ is bounded, since $b_0\alpha(\xi_n)+a_0\beta(\xi_n)\leqslant 0.$ Thus, by estimate \eqref{tscest-homo-alphabeta}, the behaviour of the term $e^{-\beta(\xi_n)M_n}$ when $|\xi_n|\rightarrow \infty$ imply that $\widehat{f}(\cdot,\xi_n)$ decays rapidly, since $M_n\rightarrow M_{ab}>0$ and $\beta(\xi_n)\geqslant \log(|\xi_n|^n).$

It follows that \[f(t,x)=\sum_{n=1}^{\infty}\widehat{f}(t,\xi_n)e^{ix\xi_n} \in C^{\infty}(\mathbb{T}^1\times\mathbb{T}^N).\]

Since $\xi_n\not\in Z_{\mathcal{M}},$ we may define
\[
\widehat{u}(t,\xi_n)=\frac{1}{1-e^{-2\pi i\mathcal{M}_0(\xi_n)}} \int_{0}^{2\pi}\exp \left(-i\int_{t-s}^{t}\mathcal{M}(r,\xi_n)dr\right) \widehat{f}(t-s,\xi_n) ds,
\]
which belongs to $C^{\infty}(\mathbb{T}^1).$

\smallskip
For $n$ large enough, the estimate
\begin{equation*}
|(1-e^{-2\pi i\mathcal{M}_0(\xi_n)})^{-1}\widehat{f}(t-s,\xi_n)|\leqslant e^{-\beta(\xi_n)M_n}\leqslant 1
\end{equation*}
implies that
\[
|\widehat{u}(t,\xi_n)|\leqslant \int_{0}^{2\pi} \exp\left(-\beta(\xi_n)\Big(M_n-\int_{t-s}^{t}a(r)+b(r)\frac{\alpha(\xi_n)}{\beta(\xi_n)}dr\Big)\right)ds \leqslant 2\pi.
\]
Hence, $\widehat{u}(\cdot,\xi_n)$ increases slowly. Then \[u=\sum_{n=1}^{\infty}\widehat{u}(t,\xi_n) e^{ix\xi_n} \in \mathcal{D}'(\mathbb{T}^1\times\mathbb{T}^N).\]

We will show that $u\not\in C^{\infty}(\mathbb{T}^1\times\mathbb{T}^N).$ In fact, for $n$ sufficiently large we have $\sigma_0+\epsilon<t_n,$ from which we can infer that
\begin{eqnarray*}
  |\widehat{u}(t_n,\xi_n)| & = & \left|\int_{t_n-\sigma_0-\epsilon}^{t_n-\sigma_0+\epsilon} \phi(t_n-s)\right.\\[2mm]
  & & \ \times \left. \exp\left(-\beta(\xi_n)\left(M_n-\int_{t_n-s}^{t_n}a(r)+b(r)\frac{\alpha(\xi_n)}{\beta(\xi_n)}dr\right)\right)ds\right|.
\end{eqnarray*}

Since $t_n-s_n\rightarrow\sigma_0,$ we have
\begin{equation*}
t_n-s_n-\sigma_0-\epsilon/2<-\epsilon/4 \ \textrm{ and } \ t_n-s_n-\sigma_0+\epsilon/2>\epsilon/4,
\end{equation*}
for $n$ large enough. Hence, for $n$ large enough, we have
\begin{equation*}
(s_n-\epsilon/4,s_n+\epsilon/4)\subset(t_n-\sigma_0-\epsilon,t_n-\sigma_0+\epsilon)
\end{equation*}
and $\phi(t_n-s)=1,$ for $s\in(s_n-\epsilon/4,s_n+\epsilon/4).$ It follows that \[
|\widehat{u}(t_n,\xi_n)| \geqslant \int_{s_n-\epsilon/4}^{s_n+\epsilon/4} \exp\left(-\beta(\xi_n) \Big(M_n-\int_{t_n-s}^{t_n}a(r)+b(r)\frac{\alpha(\xi_n)}{\beta(\xi_n)}dr\Big)\right)ds.
\]

For each $n,$ the function
\[[s_n-\epsilon/4,s_n+\epsilon/4]\ni s\mapsto\phi_n(s)\doteq M_n-\int_{t_n-s}^{t_n}a(r)+b(r)\frac{\alpha(\xi_n)}{\beta(\xi_n)}dr\]
vanishes at $s_n$ and $\phi_n(s)\geqslant 0,$ for all $s.$
Furthermore, since $\alpha(\xi_n)/\beta(\xi_n)\rightarrow K$ and $t_n\rightarrow t_0,$ there exists $C>0,$ which does not depend on $n,$ such that  \[|\widehat{u}(t_n,\xi_n)|\geqslant \int_{s_n-\epsilon/4}^{s_n+\epsilon/4}e^{-\beta(\xi_n)C(s-s_n)^2}ds.\]

The Laplace Method for Integrals implies that \[|\widehat{u}(t_n,\xi_n)|\geqslant \widetilde{C}\beta(\xi_n)^{-1/2},\] where $\widetilde{C}>0$ does not depend on $n.$

As in the proof of necessity of item $ii)$ in Theorem \ref{gt3},  the previous estimate implies that $\widehat{u}(\cdot,\xi_n)$ does not decay rapidly. Hence, $u$ belongs to $\mathcal{D}'(\mathbb{T}^1\times\mathbb{T}^N)\setminus C^{\infty}(\mathbb{T}^1\times\mathbb{T}^N)$ and $L$ is not globally hypoelliptic.

\smallskip

Finally, in the case $\beta(\xi_n)<0,$ for all $n,$ we repeat the technique above, but now we use \[M_n=\min_{0\leqslant t,s\leqslant 2\pi}\left\{\int_{t-s}^{t}a(r)+b(r)\frac{\alpha(\xi_n)}{\beta(\xi_n)}dr\right\}=\int_{t_n-s_n}^{t_n}a(r)+b(r)\frac{\alpha(\xi_n)}{\beta(\xi_n)}dr\] and \[M_{ab}=\min_{0\leqslant t,s\leqslant 2\pi}\int_{t-s}^{t}a(r)+b(r)Kdr=\int_{t_0-s_0}^{t_0}a(r)+b(r)Kdr<0.\]

The proof of Theorem \ref{ar1} is complete.
\end{proof}

\begin{corollary}\label{cornew}If $\beta(\xi)$ has super-logarithmic growth and $\alpha(\xi)/\beta(\xi)\rightarrow K,$ as $|\xi|\rightarrow\infty,$ then $L$ is not globally hypoelliptic  if $a(t)+b(t)K$ changes sign. Similarly, if $\alpha(\xi)$ has super-logarithmic growth and $\beta(\xi)/\alpha(\xi)\rightarrow C,$ as $|\xi|\rightarrow\infty,$ then $L$ is not globally hypoelliptic if $b(t)+a(t)C$ changes sign.
\end{corollary}

We say that $\beta(\xi)$ goes to infinity faster than $\alpha(\xi)$, and use the notation $\alpha(\xi)=o(\beta(\xi))$, if for all positive constant $\kappa$ there exists a positive constant $n_0$ such that
$|\alpha(\xi)|\leqslant \kappa|\beta(\xi)|,$ for all $|\xi|\geqslant n_0.$ Note that, in this case, $\alpha(\xi)/\beta(\xi)\rightarrow0,$ as $|\xi|\rightarrow\infty.$

\begin{corollary}\label{cornew2}
If $\beta(\xi)$ has super-logarithmic growth and $\alpha(\xi)=o(\beta(\xi)),$ then $L$ is not globally hypoelliptic if $a(t)$ changes sign. If $\alpha(\xi)$ has super-logarithmic growth and $\beta(\xi)=o(\alpha(\xi)),$ then $L$ is not globally hypoelliptic if $b(t)$ changes sign.
\end{corollary}

\begin{remark}The main contribution of Theorem \ref{ar1} and its corollaries is in the case where both $\alpha(\xi)$ and $\beta(\xi)$ have super-logarithmic growth. We invite the reader to compare this result with items $ii)$ and $iii)$ in Theorem \ref{gt3}.
\end{remark}

\begin{example}\label{exampnew1}If $a(t)=\cos^2(t),$ $b(t)=-\sin^2(t),$ $\alpha(\xi)=\sqrt{|\xi|}$ and $\beta(\xi)=\sqrt{|\xi|+1},$ then Theorem \ref{ar1} implies that $L$ is not globally hypoelliptic. Note that $\alpha(\xi)/\beta(\xi)\rightarrow1,$ as $|\xi|\rightarrow\infty,$ and $a(t)+b(t)=\cos^2(t)-\sin^2(t)$ changes sign.
\end{example}

Under the conditions in Corollary \ref{cornew}, in the case in which $\beta(\xi)$ has super-logarithmic growth with
\begin{equation*}
\liminf_{|\xi|\rightarrow\infty } |\xi|^M |\beta(\xi)| > 0,
\end{equation*}
for some $M\geqslant 0$ and $K\doteq\lim_{|\xi|\rightarrow\infty}\alpha(\xi)/\beta(\xi),$ the operator $L$ is globally hypoelliptic provided that $a(t)+b(t)K$ never vanishes. In fact, for $|\xi|$ sufficiently large, the function
\begin{equation*}
t\mapsto\Im\mathcal{M}(t,\xi)=\beta(\xi)[a(t)+b(t)\alpha(\xi)/\beta(\xi)]
\end{equation*}
does not change sign. Moreover, $L_0$ is globally hypoelliptic, since $|a_0+b_0K|>0$ and
\begin{equation*}
|\tau+\mathcal{M}_0(\xi)|\geqslant |\Im\mathcal{M}_0(\xi)|\geqslant |\beta(\xi)||a_0+b_0\alpha(\xi)/\beta(\xi)|\geqslant|\beta(\xi)||a_0+b_0K|/2,
\end{equation*}
for $|\xi|$ large enough. Hence, Theorem \ref{gt1} implies that $L$ is globally hypoelliptic.

\begin{example}\label{exampnew2}Assume that $a(t)=1+\sin(t)$ and $b(t)=1-\sin(t).$ If $\alpha(\xi)=\sqrt{|\xi|}+\xi$ and $\beta(\xi)=\xi,$ then $\alpha(\xi)/\beta(\xi)\rightarrow 1,$ as $|\xi|\rightarrow\infty,$ and $a(t)+b(t)=2$ never vanishes. Hence, $L$ is globally hypoelliptic.
\end{example}

On the other hand, if $a(t)+b(t)K$ does not change sign, but $a(t)+b(t)K$ vanishes, then $L$ may be non-globally hypoelliptic. In Subsection \ref{exar2} we explore this phenomenon when $K=0,$ where we present a non-globally hypoelliptic operator in the case in which $a(t)$ does not change sign, $a(t)$ vanishes (of finite order) at a singular point, both $\alpha(\xi)$ and $\beta(\xi)$ have super-logarithmic growth, and $\alpha(\xi)=o(\beta(\xi)).$

\subsection{Order of vanishing}\label{exar2}

The idea here is to show that certain relations between the order of vanishing of $a(t)$ and the speed in which $\alpha(\xi)$ and $\beta(\xi)$ go to infinity, play a role in the global hypoellipticity of the operators studied in this article.

We start with an example which illustrates that the operator may be globally hypoelliptic if, for $\xi$ large, the functions $\Im\mathcal{M}(t,\xi)$ vanishes only of finite order, and the order of vanishing at each zero is appropriated to absorb the growth of $p(\xi).$ This situation is generalized in Theorem \ref{conjar} and, in the sequence, we show that the converse of this result does not hold.

\bigskip
\noindent {\sc First example:} Let $b\equiv1$ and $a\in C^\infty(\mathbb{T}^1,\mathbb{R})$ be a function such that $a(t)=-(t-\pi)^2$ on a fixed interval $(\pi-\epsilon,\pi+\epsilon),$ $a(t)$ is increasing on $[0,\pi-\epsilon),$ and is decreasing on $(\pi+\epsilon,2\pi].$

Note that, $a^{-1}(0)\cap[0,2\pi]=\{\pi\}$ and $a(t)<0$ for $t\in[0,\pi)\cup(\pi,2\pi].$
Setting $p(\xi)=\sqrt{|\xi|}+i|\xi|\sqrt{|\xi|},$ $\xi\in\mathbb{Z},$ we have
\begin{equation*}
\Im\mathcal{M}(t,\xi)=\sqrt{|\xi|}(|\xi| a(t)+1).
\end{equation*}

Note that $\Im\mathcal{M}(\cdot,\xi)$ changes sign for all but a finite number of indexes $\xi.$

We will prove that
\begin{equation*}
L=D_t+(a+i)P(D_x), \  (t,x)\in\mathbb{T}^2,
\end{equation*}
is globally hypoelliptic, for this, given $u\in\mathcal{D}'(\mathbb{T}^2)$ such that $iLu=f,$ with $f\in C^\infty(\mathbb{T}^2),$ we must show that $u\in C^\infty(\mathbb{T}^2).$

Lemma \ref{lt1} implies that  $\widehat{u}(\cdot,\xi)$ belongs to $C^\infty(\mathbb{T}^1)$ for all $\xi,$ and for $|\xi|> -a_0^{-1},$ we may write
\[
\widehat{u}(t,\xi) = \frac{1}{ 1-e^{-2\pi i\mathcal{M}_0(\xi)}}\int_{0}^{2\pi}\exp\left(-\int_{t-s}^{t}i\mathcal{M}(r,\xi)dr\right) \widehat{f}(t-s,\xi) ds.
\]

For $|\xi|> -a_0^{-1},$ the term $(1-e^{-2\pi i\mathcal{M}_0(\xi)})^{-1}$ is bounded; indeed, since $a_0<0$ we have
\begin{equation*}
2\pi\Im\mathcal{M}_0(\xi)=2\pi\sqrt{|\xi|}(|\xi|a_0+1)\rightarrow-\infty, \ \textrm{ as} \ |\xi|\rightarrow\infty.
\end{equation*}
Moreover, for $t,s\in[0,2\pi],$ we have
\begin{align*}
-\Re\int_{t-s}^{t}i\mathcal{M}(r,\xi)dr =& \int_{t-s}^{t}\Im\mathcal{M}(r,\xi)dr \\
   = & \int_{t-s}^{t}|\xi|\sqrt{|\xi|}a(r)+\sqrt{|\xi|}dr\\
   \leqslant& \int_{\pi-1/\sqrt{|\xi|}}^{\pi+1/\sqrt{|\xi|}}|\xi|\sqrt{|\xi|}a(r)+\sqrt{|\xi|}dr\\
   = & -\int_{\pi-1/\sqrt{|\xi|}}^{\pi+1/\sqrt{|\xi|}}|\xi|\sqrt{|\xi|}(r-\pi)^2dr+2=4/3,
\end{align*}
for $|\xi|$ sufficiently large.

Hence, the rapid decaying of $\widehat{f}(\cdot,\xi)$ and estimates \eqref{tscest-homo-alphabeta} will imply that $\widehat{u}(\cdot,\xi)$ decays rapidly. Hence, $u \in C^\infty(\mathbb{T}^2)$ and  $L$ is globally hypoelliptic.

\bigskip
The following result generalizes the situation presented in the previous example.

\begin{theorem}\label{conjar} Suppose that $\beta(\xi)$ has super-logarithmic growth with
\begin{equation*}
\liminf_{|\xi|\rightarrow\infty } |\xi|^M |\beta(\xi)| > 0,
\end{equation*}
for some $M\geqslant 0$ and $\alpha(\xi)=o(\beta(\xi)).$ Assume that $a(t)$ does not change sign and vanishes of finite order only. Write $a^{-1}(0)=\{t_1<\cdots<t_n\}$ and let $m_j$ be the order of vanishing of $a(t)$ at $t_j,$ $j=1,\ldots,n.$ If for each $j$ we have
\begin{equation*}
|\alpha(\xi)/\beta(\xi)|^{1/m_j}|\alpha(\xi)|=O(\log(|\xi|)),
\end{equation*}
then the operator $L$ given by \eqref{MO} is globally hypoelliptic.
\end{theorem}

\begin{proof}Given $u\in\mathcal{D}'(\mathbb{T}^1\times\mathbb{T}^N)$ such that $iLu=f,$ with $f\in C^{\infty}(\mathbb{T}^1\times\mathbb{T}^N),$ we must show that $u\in C^{\infty}(\mathbb{T}^1\times\mathbb{T}^N).$

Without loss of generality, assume that $a(t)$ is non-negative.

By Lemma \ref{lt1}, the coefficients $\widehat{u}(\cdot,\xi)$ are smooth on $\mathbb{T}^1,$ for all $\xi\in\mathbb{Z}^N.$ Moreover, since $\Im\mathcal{M}_0(\xi)=\beta(\xi)a_0+b_0\alpha(\xi),$ $a_0>0$ and $\alpha(\xi)=o(\beta(\xi)),$ for $|\xi|$ large enough we have $\Im\mathcal{M}_0(\xi)\neq0,$ and then $\mathcal{M}_{0}(\xi)\not\in Z_\mathcal{M}.$

Hence, for $|\xi|$ sufficiently large, we may write
\[
\widehat{u}(t,\xi)= \frac{1}{1-e^{-2\pi i\mathcal{M}_0(\xi)}}\int_{0}^{2\pi}\exp\left(-\int_{t-s}^{t}i\mathcal{M}(r,\xi)dr \right) \widehat{f}(t-s,\xi) ds,
\]
if $\beta(\xi)<0,$ and
\[\widehat{u}(t,\xi)=\frac{1}{e^{2\pi i\mathcal{M}_0(\xi)}-1}\int_{0}^{2\pi}\exp\left(\int_{t}^{t+s}i\mathcal{M}(r,\xi)dr \right) \widehat{f}(t+s,\xi) ds,\]
if $\beta(\xi)>0.$

\bigskip
We must show that the sequence $\widehat{u}(\cdot,\xi)$ decays rapidly. Notice that,
\begin{align*}
|\tau + \mathcal{M}_0(\xi)| & \geqslant |\Im\mathcal{M}_0(\xi)| = |\beta(\xi)| |(a_0+b_0\alpha(\xi)/\beta(\xi))| \\
                            & \geqslant C |\xi|^{-M},
\end{align*}
when $|\xi|\rightarrow\infty$ and it follows by Proposition \ref{lt2}  that
\begin{equation*}
|1-e^{-2\pi i\mathcal{M}_0(\xi)}|^{-1}  \ \textrm{ and } \ |e^{2\pi i\mathcal{M}_0(\xi)}-1|^{-1}
\end{equation*}
have at most polynomial growth.

Now, let $I=\cup_{j=1}^{n}I_j$ be a neighborhood of $a^{-1}(0)$ such that
$$a(t)=(t-t_j)^{m_j}a_j(t), \ t \in I_j,$$
where $a_j(t)\geqslant C_j>0,$ and $m_j$ is an even number, so $a(t)$ does not change sign.

For the indexes $\xi$ such that $\beta(\xi)<0$ and $|\xi|$ is sufficiently large, we have $\beta(\xi)(a(r)+b(r)\alpha(\xi)/\beta(\xi))<0$ on $\mathbb{T}^1\setminus I.$ Moreover, if
\begin{equation*}
\beta(\xi)(a(r)+b(r)\alpha(\xi)/\beta(\xi))=0
\end{equation*}
for a certain $r$ in $I_j,$ then
\begin{equation*}
(r-t_j)^{m_j}a_j(r)=-b(r)\alpha(\xi)/\beta(\xi).
\end{equation*}

In particular,
$$
  |r-t_j|=\left|\frac{b(r)\alpha(\xi)}{a_j(r)\beta(\xi)}\right|^{1/m_j}\leqslant C_j'\left|\frac{\alpha(\xi)}{\beta(\xi)}\right|^{1/m_j},
$$
where $C_j'=(\|b\|_\infty/C_j)^{1/m_j}$.

\smallskip
It follows that $\beta(\xi)(a(r)+b(r)\alpha(\xi)/\beta(\xi))<0$ on
\begin{equation*}
\mathbb{T}^1 \setminus\bigcup_{j=1}^{n}\left[t_j-C_j'\left|\frac{\alpha(\xi)}{\beta(\xi)}\right|^{1/m_j}, \ t_j+C_j'\left|\frac{\alpha(\xi)}{\beta(\xi)}\right|^{1/m_j}\right].
\end{equation*}

Hence, for the indexes $\xi$ such that $\beta(\xi)<0$ and $|\xi|$ is sufficiently large, we obtain
\begin{align*}
& \int_{t-s}^{t}\beta(\xi)(a(r)+b(r)\alpha(\xi)/\beta(\xi))dr\leqslant \\[2mm]
& \qquad \sum_{j=1}^{n}\int_{t_j-C_j'\left|\frac{\alpha(\xi)}{\beta(\xi)}\right|^{1/m_j}}^{t_j+C_j'\left|\frac{\alpha(\xi)}{\beta(\xi)}\right|^{1/m_j}} (r-t_j)^{m_j}a_j(r)\beta(\xi)+b(r) \alpha(\xi)dr.
\end{align*}

Since
\[
\int_{t_j-C_j'\left|\frac{\alpha(\xi)}{\beta(\xi)}\right|^{1/m_j}}^{t_j+C_j'\left|\frac{\alpha(\xi)}{\beta(\xi)}\right|^{1/m_j}} \!\!\! (r-t_j)^{m_j}a_j(r)\beta(\xi)+b(r)\alpha(\xi)dr\leqslant K_j\left|\frac{\alpha(\xi)}{\beta(\xi)}\right|^{1/m_j}|\alpha(\xi)|,
\]
for some positive constant $K_j,$ and  $|\alpha(\xi)/\beta(\xi)|^{1/m_j}|\alpha(\xi)|=O(\log(|\xi|))$ (by hypothesis), it follows that  there exists a positive constant
$\widetilde{M}$ such that
\[\left|\exp\left(\!-\!\!\int_{t-s}^{t}\!\!\! i\mathcal{M}(r,\xi)dr\right)\right|=\exp\left(\int_{t-s}^{t}\!\!\! \beta(\xi)\Big(a(r)+b(r)\frac{\alpha(\xi)}{\beta(\xi)}\Big)dr\right) \leqslant |\xi|^{\widetilde{M}},\]
for all the indexes $\xi$ such that $\beta(\xi)<0$ and $|\xi|$ is sufficiently large.

This same procedure may be used to verify that a similar estimate holds true for the indexes $\xi$ such that $\beta(\xi)>0$ and $|\xi|$ is sufficiently large.

Finally, by using estimates above and \eqref{tscest-homo-alphabeta}, we may verify that the rapid decaying of $\widehat{f}(\cdot,\xi)$ will imply that $\widehat{u}(\cdot,\xi)$ decays rapidly.
Therefore $u\in C^{\infty}(\mathbb{T}^1\times\mathbb{T}^N)$ and $L$ is globally hypoelliptic.

\end{proof}

\begin{remark}
We have a similar result when $\alpha(\xi)$ has super-logarithmic growth,
\begin{equation*}
\liminf_{|\xi|\rightarrow\infty } |\xi|^M |\alpha(\xi)| > 0,
\end{equation*}
$\beta(\xi)=o(\alpha(\xi)),$ and $b(t)$ does not change sign and vanishes only of finite order.
\end{remark}

In the next example we show that the converse of Theorem \ref{conjar} does not hold true.

\bigskip
\noindent {\sc Second example:} Consider $a\in C^\infty(\mathbb{T}^1,\mathbb{R})$ as in the first example in this subsection. We will see that $L=D_t+(a(t)+i)P(D_x),$ where $(t,x)\in\mathbb{T}^2$ and $p(\xi)=\xi+i{\xi}^2,$ is not globally hypoelliptic. Notice that $a(t)$ does not change sign, but
\begin{equation*}
\Im\mathcal{M}(t,\xi)=\xi^2a(t)+\xi
\end{equation*}
changes sign for infinitely many indexes $\xi\in\mathbb{Z}.$

For $\xi>0$ large enough, we have $\xi^2a(t)+\xi<0$ on
\begin{equation*}
[0,\pi-1/\sqrt{\xi})\cup(\pi+1/\sqrt{\xi},2\pi]
\end{equation*}
and $\xi^2a(t)+\xi=-\xi^2(t-\pi)^2+\xi>0$ on $(\pi-1/\sqrt{\xi},\pi+1/\sqrt{\xi}),$ so that \[M_\xi\doteq\max_{t,s\in[0,2\pi]}\int_{t-s}^{t}\!\!\!\Im\mathcal{M}(r,\xi)dr=\int_{\pi-1/\sqrt{\xi}}^{\pi+1/\sqrt{\xi}}\!\!\!-\xi^2(r-\pi)^2+\xi dr=4\sqrt{\xi}/3.\]

Let $\widehat{f}(\cdot,\xi)$ be the $2\pi-$periodic extension of
\[
(1-e^{-2\pi i\mathcal{M}_0(\xi)})\exp\left(i\int_{t}^{\pi+1/\sqrt{\xi}}\Re\mathcal{M}(r,\xi)dr\right) e^{-M_\xi}\phi_\xi(t),\] in which $\phi_\xi\in C^{\infty}_c((\pi-2/\sqrt{\xi},\pi),\mathbb{R})$ is given by $\phi_\xi(t)\doteq \psi(\sqrt{\xi}(t-\pi+1/\sqrt{\xi})),$ with $\psi\in C^\infty_c((-1,1),\mathbb{R}),$ $0\leqslant \psi\leqslant 1,$ and $\psi\equiv1$ in a neighborhood of $[-1/2,1/2].$

Notice that $1-e^{-2\pi i\mathcal{M}_0(\xi)}$ is bounded, since $a_0<0$ implies that
\begin{equation*}
\Im\mathcal{M}_0(\xi)=\xi^2a_0+\xi<0,
\end{equation*}
for $\xi$ large enough.

With these definitions, by using \eqref{tscest-homo-alphabeta} we may see that $\widehat{f}(\cdot,\xi)$ decays rapidly. Hence
\[f(t,x)\doteq\sum_{\xi>-a_0^{-1}}\widehat{f}(t,\xi)e^{ix\xi} \in C^\infty(\mathbb{T}^2).\]

In order to exhibit $u\in\mathcal{D}'(\mathbb{T}^2)\setminus C^{\infty}(\mathbb{T}^2)$ such that $iLu=f,$ consider

\[
\widehat{u}(t,\xi)=\frac{1}{1-e^{-2\pi i\mathcal{M}_0(\xi)}}\int_{0}^{2\pi}\exp\left(-\int_{t-s}^{t}i\mathcal{M}(r,\xi)dr \right) \widehat{f}(t-s,\xi) ds,\]
for $\xi>-a_0^{-1}$.

Note that $1-e^{-2\pi i\mathcal{M}_0(\xi)}\neq0;$ hence $\widehat{u}(\cdot,\xi)$ is well defined and belongs to $C^\infty(\mathbb{T}^1)$.

For $s,t\in[0,2\pi],$ we have
\begin{align*}
 & \left|\frac{1}{1-e^{-2\pi i\mathcal{M}_0(\xi)}}\widehat{f}(t-s,\xi)\exp\left(-\int_{t-s}^{t}i\mathcal{M}(r,\xi)dr\right)\right|\leqslant \\[2mm]
 & \qquad \|\psi\|_{\infty}\exp\left(-\Big(M_\xi-\int_{t-s}^{t}\Im\mathcal{M}(r,\xi)dr\Big)\right)\ \leqslant \ 1.
\end{align*}
Thus $|\widehat{u}(t,\xi)|\leqslant 2\pi,$ which implies that $\widehat{u}(\cdot,\xi)$ increases slowly. It follows that
\[u=\sum_{\xi>-a_0^{-1}}\widehat{u}(t,\xi)e^{ix\xi} \in \mathcal{D}'(\mathbb{T}^2),\]
and Lemma \ref{lt1} implies that $iLu=f.$

\bigskip
Finally,
\begin{eqnarray*}
  |\widehat{u}(\pi+1/\sqrt{\xi},\xi)| & = & \int_{1/\sqrt{\xi}}^{3/\sqrt{\xi}}\phi_\xi(\pi+1/\sqrt{\xi}-s) \\
  & & \times \exp\left(\!\!-\Big(M_\xi-\!\!\int_{\pi+1/\sqrt{\xi}-s}^{\pi+1/\sqrt{\xi}}\!\!\Im\mathcal{M}(r,\xi)dr\Big)\!\!\right)ds.
\end{eqnarray*}

Since $2/\sqrt{\xi}$ is a zero of order at least two of
\[
\theta_\xi(s)\doteq M_\xi-\int_{\pi+1/\sqrt{\xi}-s}^{\pi+1/\sqrt{\xi}}\Im\mathcal{M}(r,\xi)dr\geqslant 0,
\]
it follows that
$$\theta_\xi(s)\leqslant  (s-2/\sqrt{\xi})^2\|\theta_\xi''\|_\infty\leq(s-2/\sqrt{\xi})^2\xi^2(\|a\|_{\infty}+1).$$
Hence
\begin{align*}|\widehat{u}(\pi+1/\sqrt{\xi},\xi)|\geqslant &\int_{1/\sqrt{\xi}}^{3/\sqrt{\xi}}\psi(2-s\sqrt{\xi})e^{-\xi^2(\|a\|_{\infty}+1)(s-2/\sqrt{\xi})^2}ds\\
\geqslant &\int_{3/(2\sqrt{\xi})}^{5/(2\sqrt{\xi})}e^{-\xi^2(\|a\|_{\infty}+1)(s-2/\sqrt{\xi})^2}ds\\
=&\int_{-1/(2\sqrt{\xi})}^{1/(2\sqrt{\xi})}e^{-\xi^2(\|a\|_{\infty}+1)s^2}ds.
\end{align*}

As previously mentioned, the Laplace Method for Integrals implies that
\begin{equation*}
|\widehat{u}(\pi+1/\sqrt{\xi},\xi)|\geqslant K/\xi,
\end{equation*}
where $K$ is a positive constant which does not depend on $\xi.$ In particular, $\widehat{u}(\cdot,\xi)$ does not decay rapidly and $L$ is not globally hypoelliptic.

\section{Homogeneous operators}\label{sectsc}

In the previous section we saw that, (in general) the converse of Theorem \ref{gt1} does not hold, since there exist globally hypoelliptic operators of type \eqref{MO} for which the function $t\in\mathbb{T}^1\rightarrow\Im\mathcal{M}(t,\xi)$ changes sign, for infinitely many indexes $\xi.$

We present here a class of symbols where the converse holds. For instance, if $p(\xi)$ is  homogeneous of order one, then this converse holds, since, in this case, condition $(P)$ of Nirenberg-Treves is necessary for the global hypoellipticity (see \cite{HORV4}, Corollary 26.4.8).

We will see that the converse of Theorem \ref{gt1} holds true in the case in which $p(\xi)$ is  homogeneous of any positive degree.

In the sequel, we present a class of operators composed of a sum of  homogeneous pseudo-differential operators, for which the study of the global hypoellipticity follows from the techniques used in this article.

\begin{theorem}\label{corolph}Assume that the symbol of $P(D_x)$ is  homogeneous of degree $m.$
\begin{enumerate}
  \item[$i)$]  If $m\leqslant 0$ then $L$ is globally hypoelliptic if and only if $L_0$ is globally hypoelliptic;
  \item[$ii)$] If $m>0$  then $L$ is globally hypoelliptic if and only if $L_0$ is globally hypoelliptic, and the function $t\mapsto\Im\mathcal{M}(t,\xi)$ does not change sign, for all $\xi\in\mathbb{Z}^N\setminus\{0\}.$
\end{enumerate}

\end{theorem}

\begin{proof}If $m\leqslant 0,$ the result follows from item $i)$ of Theorem \ref{gt3}. For the case in which $m>0,$ the presented conditions are sufficient thanks to Theorem \ref{gt1}. On the other hand, if there exists $\xi_0\in\mathbb{Z}^N\setminus\{0\}$ such that $t\mapsto\Im\mathcal{M}(t,\xi_0)$ changes sign, then
\begin{equation*}
t\mapsto (n|\xi_0|)^m\Im\mathcal{M}(t,\xi_0/|\xi_0|)
\end{equation*}
changes sign for all $n\in\mathbb{N}.$

Now in order to show that $L$ is not globally hypoelliptic, we may repeat the techniques in the proof of the necessity in item $ii)$ of Theorem \ref{gt3}.

\end{proof}

The following result is a consequence of Theorem \ref{gt0} and Theorem \ref{corolph}.

\begin{corollary}\label{inttheor}
Let $p=p(\xi)$ be a  homogeneous symbol of degree $m=\ell/q,$ with $\ell,q\in\mathbb{N},$ and $\gcd(\ell,q)=1$. Write $p(1)=\alpha+i\beta$ and $p(-1)=\widetilde{\alpha}+i\widetilde{\beta}.$ The operator \[L=D_t+(a+ib)(t)P(D_x),\quad (t,x)\in\mathbb{T}^2,\] is globally hypoelliptic if and only if  the following statements occur:
\begin{itemize}
\item[i)] the functions $t\in\mathbb{T}^1\mapsto a(t)\beta+b(t)\alpha$ and $t\in\mathbb{T}^1\mapsto a(t)\widetilde{\beta}+b(t)\widetilde{\alpha}$ do not change sign.
\item[ii)]   $(a_0\alpha-b_0\beta)^q$ is an irrational non-Liouville number whenever $a_0\beta+b_0\alpha=0$, and $(a_0\widetilde{\alpha}-b_0\widetilde{\beta})^q$ is an irrational non-Liouville number whenever $a_0\widetilde{\beta}+b_0\widetilde{\alpha}=0.$
\end{itemize}
\end{corollary}

\subsection{Sum of homogeneous operators}

The techniques used in this article allow us to study the global hypoellipticity of operators of the type
\begin{equation}\label{LND}
L=D_t+\sum_{j=1}^{N}(a_j+ib_j)(t)P_j(D_{x_j}),\,\,\,(t,x)\in\mathbb{T}^1\times\mathbb{T}^N,
\end{equation}
where each $P_j(D_{x_j})$ is  homogeneous  of degree $m_j$ (see Definition \ref{defphs}), so that its symbol $p_j(\xi_j)$ satisfies
\begin{equation*}
p_j(\xi_j)=
\left\{
\begin{array}{ll}
 \ \xi_j^{m_j}p_j(1), & \textrm{ if } \ \xi_j> 0, \\[3mm]
|\xi_j|^{m_j}p_j(-1), & \textrm{ if } \ \xi_j< 0.
\end{array}
\right.
\end{equation*}

The results presented in this subsection generalize Theorem 1.3 of \cite{BDG}, see Corollary \ref{lastcor} below.

The constant coefficient operator $L_0$ associated to the operator $L$ given by \eqref{LND} is
\[L_0=D_t+\sum_{j=1}^{N}(a_{j0}+ib_{j0})P_j(D_{x_j}).\]

We also set, for $j=1,\ldots,N,$
\begin{eqnarray*}
   \mathcal{M}_j(t,\xi_j) \ \doteq \ (a_j+ib_j)(t)p_j(\xi_j), & & \mathcal{M}_{j0}(\xi_j) \ \doteq \ (a_{j0}+ib_{j0})p_j(\xi_j) \\[3mm]
   \mathcal{M}(t,\xi) \ \doteq \ \sum_{j=1}^N\mathcal{M}_j(t,\xi_j), \mbox{ and} & &  \mathcal{M}_{0}(\xi) \ \doteq \ \sum_{j=1}^N\mathcal{M}_{j0}(\xi_j).
\end{eqnarray*}

\begin{theorem}\label{conjlnd}
The operator $L$ given by \eqref{LND} is globally hypoelliptic if the following situations occur:
\begin{itemize}
\item[i)] $L_0$ is globally hypoelliptic.

\item[ii)] for each pair $j,k\in\{1,\ldots,N\}$ ($j\neq k$) such that $m_j>0$ and $m_k>0,$ the sets of real-valued functions
$$\Upsilon_{r,s} \doteq \{\Im\mathcal{M}_j(\cdot,(-1)^r),\Im\mathcal{M}_k(\cdot,(-1)^s)\}, \ r,s\in\{1,2\},$$
are $\mathbb{R}-$linearly dependent.

\item[iii)] for each $\xi_j\in\mathbb{Z}\setminus\{0\},$ the function $t\in\mathbb{T}^1\rightarrow \Im\mathcal{M}_j(t,\xi_j)$ does not change sign whenever $m_j>0,$ $j=1,\ldots,N.$
\end{itemize}

On the other hand, if $L$ is globally hypoelliptic, then conditions $i)$ and $iii)$ hold.
\end{theorem}

We notice that $L$ may be non-globally hypoelliptic if conditions $i)$ and $iii)$ hold, but condition $ii)$ fails. For instance, consider the operator
\begin{equation*}
D_t+i\cos^2(t)D_{x_1}+i\sqrt{2}\sin^2(t)D_{x_2}, \  (t,x_1,x_2)\in\mathbb{T}^3.
\end{equation*}
This operator satisfies $i)$ and $iii),$ but $ii)$ fails, since $\cos^2(t)$ and $\sqrt{2}\sin^2(t)$ are $\mathbb{R}-$linearly independent functions. Theorem 1.3 of \cite{BDGK} implies that this operator is not globally hypoelliptic.

Before presenting the proof of Theorem \ref{conjlnd}, we give an example which shows that condition $ii),$ in general, is not necessary for the global hypoellipticity of $L.$

\begin{example}
Consider \[L=D_t+i\cos^2(t)D^2_{x_1}+i\sin^{2}(t)D^2_{x_2},\,\,\,(t,x_1,x_2)\in\mathbb{T}^3.\] Note that $\Im\mathcal{M}_1(t,1)=\cos^2(t)$ and $\Im\mathcal{M}_2(t,1)=\sin^2(t)$ are $\mathbb{R}-$linearly independent functions. Moreover, condition $iii)$ is satisfied and we have
\begin{equation*}
|\tau+i\Im\mathcal{M}_{10}(\xi_1)+i\Im\mathcal{M}_{20}(\xi_2)|=|\tau+i(\xi_1^2+\xi_2^2)/2|\geqslant1/2,
\end{equation*}
for all $(\xi_1,\xi_2)\in\mathbb{Z}^2\setminus\{(0,0)\}.$ Hence, condition $i)$ is also satisfied.

By using partial Fourier series in the variables $(x_1,x_2)$ and proceeding as in the proof of Theorem \ref{gt1}, we see that $L$ is globally hypoelliptic.
\end{example}

\begin{proof}[Sketch of the proof of Theorem \ref{conjlnd}]

\bigskip

\noindent{\textit{Sufficient conditions:}}

\medskip

Given a distribution $u\in\mathcal{D}'(\mathbb{T}_t^1\times\mathbb{T}^N_x)$ such that $iLu=f,$ with $f\in C^{\infty}(\mathbb{T}^1\times\mathbb{T}^N),$ we must show that $u\in C^{\infty}(\mathbb{T}^1\times\mathbb{T}^N).$

By using partial Fourier series in the variable $x=(x_1,\ldots,x_N),$ we are led to the equations \[\partial_t\widehat{u}(t,\xi)+i\widehat{u}(t,\xi)\sum_{j=1}^{N}\mathcal{M}_j(t,\xi_j)=\widehat{f}(t,\xi),\,\,\,t\in\mathbb{T}^1,\,\,\,\xi=(\xi_1,\ldots,\xi_N)\in\mathbb{T}^N.\]

Since $L_0$ is globally hypoelliptic, proceeding as in the proof of Proposition \ref{tt1} we see that \[Z_{\mathcal{M}}=\left\{\xi\in\mathbb{Z}^N;\,\sum_{j=1}^N\mathcal{M}_{j0}(\xi_j)\in\mathbb{Z}\right\}\] is finite. Hence, Lemma \ref{lt1} implies that for all but a finite number of indexes $\xi,$ $\widehat{u}(t,\xi)$ is written as either
\begin{align}
\widehat{u}(t, \xi) = \frac{1}{1 - e^{-  2 \pi i{\mathcal{M}_0}(\xi)}} \int_{0}^{2\pi}\exp\left(-i\int_{t-s}^{t}{\mathcal{M}}(r, \, \xi) \, dr\right) \widehat{f}(t-s, \xi)ds, \label{Solu-1f}
\end{align}
or
\begin{align}
\widehat{u}(t, \xi) = \frac{1}{e^{ 2 \pi i{\mathcal{M}_0}(\xi)}-1} \int_{0}^{2\pi}\exp\left(i\int_{t}^{t+s}{\mathcal{M}}(r, \, \xi) \, dr \right) \widehat{f}(t+s, \xi)ds, \label{Solu-2f}
\end{align} where now \[\mathcal{M}(t,\xi)=\sum_{j=1}^N\mathcal{M}_j(t,\xi_j).\]

Assume that $m_{j}>0,$ for $j=1,\ldots,r,$ and $m_{j}\leqslant0,$ for $j=r+1,\ldots,N.$

From formulas \eqref{Solu-1f} and \eqref{Solu-2f} we see that, in order to show that $\widehat{u}(\cdot,\xi)$ decays rapidly, it is enough to control the imaginary part of the functions \[t\in\mathbb{T}^1\rightarrow\sum_{j=1}^{r}\mathcal{M}_j(t,\xi_j).\] Recall that the global hypoellipticity of $L_0$ implies that the sequences $(1 - e^{\pm 2 \pi i{\mathcal{M}_0}(\xi)} )^{-1}$ increases slowly (Proposition \ref{lt2}).

For the indexes $\xi\in\mathbb{Z}^N$ such that $\xi_1=\cdots=\xi_r=0,$ we have \[\sum_{j=1}^{r}\Im\mathcal{M}_j(t,\xi_j)=\sum_{j=1}^{r}\Im\mathcal{M}_j(t,0),\] which does not depend on $\xi.$

Suppose now that $\xi\in\mathbb{Z}^N$ is such that $\xi_{1}\neq0.$ Since
\begin{equation*}
\Im\mathcal{M}_j(t,\xi_j)=|\xi_j|^{m_j}\Im\mathcal{M}_j(t,\pm1), \ \xi_j\neq0,
\end{equation*}
under assumption $ii)$ it follows that
\[\sum_{j=1}^{r}\Im\mathcal{M}_j(t,\xi_j)=\Im\mathcal{M}_{1}(t,\pm1)\sum_{j=1}^{r}\lambda_{j}^{\pm}|\xi_j|^{m_j}+\sum_{j=1}^{r}\gamma_j\Im\mathcal{M}_j(t,0),\] where $\lambda_j^{\pm}$ and $\gamma_j$ are real numbers, $j=1,\ldots,r.$ Moreover, $\gamma_j=0$ when $\xi_j\neq0,$ and $\lambda^{\pm}_j=0$ when $\xi_j=0.$

An analogous formula holds if at least one $\xi_j$ is non-zero, for $j=1,\ldots,r.$

We note that we have a finite number of such formulas which we may use to represent \[\sum_{j=1}^{r}\Im\mathcal{M}_j(t,\xi_j),\] for all indexes $\xi$ such that at least one $\xi_j\neq0,$ $j=1,\ldots,r.$

By using these formulas and condition $iii)$ we may see that the rapid decaying of $\widehat{f}(\cdot,\xi)$ implies that $\widehat{u}(\cdot,\xi)$ decays rapidly (similar to which was done in the proof of item $ii)$ of Theorem \ref{gt3}).

Therefore, conditions $i)-iii)$ imply that $L$ is globally hypoelliptic.

\bigskip

\noindent{\textit{Necessary conditions:}}

Proceeding as in the proof of Theorem \ref{ncm2}, where now
\[\mathcal{M}(t,\xi)=\sum_{j=1}^{N}\mathcal{M}_j(t,\xi_j),\]
we see that condition $i)$ is necessary.

The necessity of condition $iii)$ follows from Theorem \ref{corolph}. Indeed, if
\[L_j\doteq D_t+(a_j+ib_j)(t)P_j(D_{x_j}),\,\,(t,x_j)\in\mathbb{T}^2,\]
is not globally hypoelliptic, there exists $\nu\in\mathcal{D}'(\mathbb{T}^2_{(t,x_j)})\setminus C^{\infty}(\mathbb{T}^2)$ such that $L_j\nu\in C^{\infty}(\mathbb{T}^2).$

Setting $x'=(x_1,\ldots,x_{j-1},x_{j+1},\ldots,x_N),$ it follows that
\begin{equation*}
\mu\doteq\nu\otimes 1_{x'} \in \mathcal{D}'(\mathbb{T}^2\times\mathbb{T}^{N-1})\setminus C^{\infty}(\mathbb{T}^2\times\mathbb{T}^{N-1})
\end{equation*}
and $\mu$ satisfies $L\mu=L_j\nu\in C^{\infty}(\mathbb{T}^1\times\mathbb{T}^N).$ Hence, $L$ is not globally hypoelliptic if condition $iii)$ fails.

\hfill $\square$
\end{proof}

It follows from Theorem 1.3 of \cite{BDG} that condition $ii)$ of Theorem \ref{conjlnd} is necessary if $P_j(D_{x_j})=D_{x_j},$ $j=1,\ldots,N.$ The next result gives a larger class of operators for which this necessity still holds true.

\begin{theorem}
Assume that the operator $L$ defined in \eqref{LND} is globally hypoelliptic.  Then, for all $j,k \in \{1,\ldots,N\}$ such that
$j \neq k, \ m_{j}, m_{k} \in \mathbb{Z}_+^*$ and  $\gcd(m_{j},m_{k})=1$, the sets
$$\Upsilon_{r,s} = \{\Im\mathcal{M}_j(\cdot,(-1)^r),\Im\mathcal{M}_k(\cdot,(-1)^s)\}, \ r,s\in\{1,2\},$$
are $\mathbb{R}-$linearly dependent.
\end{theorem}

\begin{proof} Let $m_1$ and $m_2$ be positive integers such that $gcd(m_1,m_2)=1$ and assume that $\Im\mathcal{M}_1(\cdot,1)$ and $\Im\mathcal{M}_2(\cdot,1)$ are $\mathbb{R}-$linearly independent functions in
 $C^{\infty}(\mathbb{T}^1,\mathbb{R})$ (the other possibilities are analogous).

By Lemma 3.1 of \cite{BDGK} there exist non-zero integers $p\neq q$ such that
\begin{equation*}
t\mapsto \Im\mathcal{M}_1(t,1) p + \Im\mathcal{M}_2(t,1) q
\end{equation*}
changes sign and has non-zero mean.

Inspired by \eqref{divinmult}, we multiply this function by
\begin{equation*}
p^{(m_1-1)m_2+(m_2-1)\ell_2m_2}q^{\ell_1m_1},
\end{equation*}
where $\ell_1$ and $\ell_2$ are non-negative integers such that $\ell_2m_2-\ell_1m_1=1.$ Hence, the function
\[t\mapsto [p^{\ell_1(m_2-1)+m_2}q^{\ell_1}]^{m_1}\Im\mathcal{M}_1(t,1)+[p^{m_1-1+(m_2-1)\ell_2}q^{\ell_2}]^{m_2}\Im\mathcal{M}_2(t,1),\]
changes sign.

Setting $\tilde{p}=p^{\ell_1(m_2-1)+m_2}q^{\ell_1}$ and $\tilde{q}=p^{m_1-1+(m_2-1)\ell_2}q^{\ell_2},$ it follows that $\tilde{p}$ and $\tilde{q}$ are integers and
\begin{align*}n^{m_1m_2}[\tilde{p}^{m_1}\Im\mathcal{M}_1(t,1)+\tilde{q}^{m_2}\Im\mathcal{M}_2(t,1)] =
\Im\mathcal{M}_1(t,\tilde{p}n^{m_2})+\Im\mathcal{M}_2(t,\tilde{q}n^{m_1})]&.
\end{align*}

Notice that, changing the variable $t$ by $-t$ and considering $-L$ (if necessary), we may assume that $\Im\mathcal{M}_{10}(\tilde{p})+\Im\mathcal{M}_{20}(\tilde{q})<0.$

We then proceed as in the proof of necessity in item $ii)$ of Theorem \ref{gt3} in order to show that
\begin{equation*}
L_{12}\doteq D_t+(a_1+ib_1)P_1(D_{x_1})+(a_2+ib_2)(t)P_2(D_{x_2}), \ (t,x_1,x_2)\in\mathbb{T}^3,
\end{equation*}
is not globally hypoelliptic. As before, this implies that $L$ is not globally hypoelliptic.

To be more precise, the technique to show that $L_{12}$ is not globally hypoelliptic consists of using the change of sign of  $n^{m_1m_2}[\Im\mathcal{M}_1(t,\tilde{p})+\Im\mathcal{M}_2(t,\tilde{q})]$ to construct a smooth function \[\hat{f}(t,x_1,x_2)=\sum_{n=1}^{\infty}\widehat{f}(t,\tilde{p}n^{m_2},\tilde{q}n^{m_1})e^{i(\tilde{p}n^{m_2}x_1+\tilde{q}n^{m_1}x_2)}\] such that $iLu=f$ has a solution in $\mathcal{D}'(\mathbb{T}^3)\setminus C^{\infty}(\mathbb{T}^3).$ The Fourier coefficients $\widehat{f}(\cdot,\tilde{p}n^{m_2},\tilde{q}n^{m_1})$ are the $2\pi-$periodic extension of
\[
\Theta_{1,2}(n) \, \phi(t)\exp\left(i\int_{t}^{t_0}\Re{M}_1(r,\tilde{p}n^{m_2})+\Re{M}_2(r,\tilde{q}n^{m_1})dr\right)
e^{-n^{m_1m_2}M},\]
where
\[ \Theta_{1,2}(n) \doteq 1-e^{-2\pi i[\mathcal{M}_{10}(\tilde{p}n^{m_2})+\mathcal{M}_{20}(\tilde{q}n^{m_1})]}\]
and
\[M=\max_{0\leqslant t,s\leqslant2\pi}\int_{t-s}^{t}\Im\mathcal{M}_1(r,\tilde{p})+\Im\mathcal{M}_2(r,\tilde{q})dr,\]
which is supposed to be assumed in $t=t_0$ and $s=s_0,$ and  $\phi$ is a smooth cutoff function identically one in a small neighborhood of $t_0-s_0.$

\end{proof}

\begin{corollary}\label{lastcor}Suppose that each symbol $p_j(\xi_j)$ is real-valued and homogeneous, whose degree is a positive integer $m_j.$ Assume also that
\begin{equation*}
\gcd(m_j,m_k)=1, \ \textrm{ for } \ j\neq k, \ j,\, k\in\{1,\ldots,N\}.
\end{equation*}

Under these assumptions, $L$ given by \eqref{LND} is globally hypoelliptic if and only if the following occurs:
\begin{itemize}
\item[i)] $L_0$ is globally hypoelliptic.

\item[ii)] $\dim{\rm span}\{b_j\in C^{\infty}(\mathbb{T}^1,\mathbb{R});\,\,j=1,\ldots,N\}\leqslant 1$

\item[iii)] $b_j(t)$ does not change sign, for $j=1,\ldots,N.$

\end{itemize}
\end{corollary}


\end{document}